\documentclass[12pt]{article}
\usepackage{amsthm,amssymb,amsmath,amsxtra,amsfonts}
\usepackage{graphicx,psfrag,epsf}
\usepackage{enumerate}
\usepackage[longnamesfirst,round]{natbib}
\usepackage{dsfont}
\usepackage{subcaption}
\usepackage{url} 
\usepackage{threeparttable}
\usepackage{booktabs}
\usepackage[hypertexnames=false,colorlinks=true,citecolor=Blue,urlcolor=Blue,linkcolor=Blue]{hyperref}
\usepackage[usenames,dvipsnames,table]{xcolor}
\usepackage{threeparttable}
\usepackage{doi}
\usepackage{mathrsfs}  
\usepackage{dsfont}
\usepackage{bbm}

\newcommand{\blind}{0}

\theoremstyle{theorem}
\newtheorem{theorem}{Theorem}[section]
\newtheorem{proposition}{Proposition}[section]
\newtheorem{lemma}{Lemma}[section]
\newtheorem{corollary}{Corollary}[section]
\newtheorem{example}{Example}[section]

\newtheorem{assumption}{Assumption}[section]

\newcommand{\R}{\ensuremath{\mathbf{R}}}

\newcommand{\Nn}{\ensuremath{\mathbf{N}}}
\newcommand{\E}{\ensuremath{\mathds{E}}}
\newcommand{\dx}{\ensuremath{\mathrm{d}}}

\newcommand{\si}{\perp \! \! \! \perp}
\DeclareMathOperator*{\argmin}{arg\,min}
\newcommand{\one}{\ensuremath{\mathbbm{1}}}

\begin{document}

	\def\spacingset#1{\renewcommand{\baselinestretch}%
		{#1}\small\normalsize} \spacingset{1}

	
	\if0\blind
	{
		\title{\bf Are Unobservables Separable?}
		\author{Andrii Babii\thanks{Department of Economics, University of North Carolina--Chapel Hill - Gardner Hall, CB 3305
				Chapel Hill, NC 27599-3305. Email: \href{mailto:babii.andrii@gmail.com}{babii.andrii@gmail.com}.}  \\
			\textit{\normalsize UNC Chapel Hill} \and Jean-Pierre Florens\thanks{Department of Economics, Toulouse School of Economics - 1, Esplanade de l'Universit\'{e}, 31080 Toulouse Cedex 06, France.} \\ \textit{\normalsize Toulouse School of Economics}}
		\maketitle
	} \fi
	
	\if1\blind
	{
		\bigskip
		\bigskip
		\bigskip
		\begin{center}
			{\LARGE\bf Are Unobservables Separable?}
		\end{center}
		\medskip
	} \fi
	
	\thispagestyle{empty}
	
	\bigskip
	\begin{abstract}
		\noindent It is common to assume in empirical research that observables and unobservables are additively separable, especially, when the former are endogenous. This is done because it is widely recognized that identification and estimation challenges arise when interactions between the two are allowed for. Starting from a nonseparable IV model, where the instrumental variable is independent of unobservables, we develop a novel nonparametric test of separability of unobservables. The large-sample distribution of the test statistics is nonstandard and relies on a novel Donsker-type central limit theorem for the empirical distribution of nonparametric IV residuals, which may be of independent interest. Using a dataset drawn from the 2015 US Consumer Expenditure Survey, we find that the test rejects the separability in Engel curves for most of the commodities.
	\end{abstract}
	
	\noindent%
	
	\setcounter{page}{0}
	\newpage
	\spacingset{1.45} 

	\section{Introduction}
	It is common to assume in empirical research that observables and unobservables are additively separable, especially when the former are endogenous. This is done because it is widely recognized that identification and estimation challenges arise when interactions between the two are allowed for. However, the economic theory and considerations often lead to nonseparable models. Prominent examples are demand functions, where the price or income effects might be heterogeneous in unobserved preferences; production functions, where observed input choices may be heterogeneous in input choices unobserved by the econometrician; labor supply functions with heterogeneous wage effects; wage equations, where the returns to schooling might vary with unobserved ability; or more generally, treatment effect models, where causal effects are heterogeneous in unobservables.
	
	In response to these empirical challenges, there is a growing literature studying the nonparametric identification of nonseparable models with endogeneity; see \cite{chernozhukov2005iv}, \cite{chernozhukov2007instrumental}, \cite{florens2008identification}, \cite{imbens2009identification}, \cite{torgovitsky2015identification}, and \cite{d2015identification} among many others. It is well-understood that the fully nonparametric estimation of a nonseparable model may lead to a difficult nonlinear ill-posed inverse problem; see \cite{carrasco2007linear}, \cite{horowitz2007nonparametric}, \cite{gagliardini2012tikhonov}, and \cite{dunker2014iterative}.
	
	Since a fully nonparametric estimation of a nonseparable model is more challenging and since separable models rule out the heterogeneity of marginal effects in unobservables, detecting separability is desirable in empirical applications. If the separability is rejected, then the more sophisticated nonseparable models should not be neglected, while if it turns out that the structural relation is separable, then the conventional empirical practice could be well-justified. 
	
	Despite the significant efforts focused on understanding the identification and the estimation of nonseparable IV models and the widespread use of separable IV models in empirical practice, little work has been done on developing formal testing procedures that could discriminate empirically between the two. \cite{lu2014testing} and \cite{su2015testing} are notable exceptions that develop separability tests under the \textit{conditional independence restriction} and additional identifying restrictions imposed by the nonseparable model. The conditional independence restriction is different from the \textit{mean-independence restriction} imposed by the separable nonparametric IV model and does not allow justifying the separable nonparametric IV model that we are interested in here. Other recent specification tests for the nonseparable model include the monotonicity test of \cite{hoderlein2014testing}, the endogeneity test of \cite{feve2013estimation}, and the specification test for the quantile IV regression of \cite{breunig2013specification}.
	
	In this paper, we design a novel fully nonparametric separability test. Our test is based on the independence condition of the nonseparable model and does not rely on additional identifying restrictions, such as the monotonicity in unobservable. The test is based on the insight that the structural function in the separable model can be estimated using the nonparametric IV approach; see \cite{florens2003inverse}, \cite{newey2003instrumental}, \cite{hall2005nonparametric}, \cite{blundell2007semi}, and \cite{darolles2011nonparametric}. If the separable model is correct, then the nonparametric IV residuals should approximate unobservables that are independent of the instrumental variables in the nonseparable IV model. This intuition suggests that it should be possible to detect the separability with the classical Kolmogorov-Smirnov or Cram\'{e}r-von Mises independence tests between the nonparametric IV residuals and the instrumental variable. To the best of our knowledge, no such test is currently available in the literature, and it is not known whether the empirical distribution of the nonparametric IV residuals satisfies the Donsker property.
	
	Formalizing this intuition is far from trivial since the regression residuals are different from the true regression errors and the nonparametric IV regression is an example of a linear \textit{ill-posed inverse} problem and requires regularization. Moreover, the empirical distribution function of the nonparametric IV residuals is a \textit{non-smooth} function of the estimated nonparametric IV regression. The weak convergence of the empirical distribution of regression residuals in the parametric linear case is a classical problem in statistics; see, e.g., \cite{durbin1973weak}, \cite{loynes1980empirical}, and \cite{mammen1996empirical}. The extension to the nonparametric regression is more challenging, and it is remarkable that the empirical distribution of nonparametric regression residuals still converges weakly as was shown in \cite{akritas2001non}. The additively separable nonparametric IV regression differs from the problems discussed above in two important directions. First, its finite-sample and the asymptotic performance depend both on the smoothness of the regression function and the smoothing properties of the conditional expectation operator. Second, it features an additional dependence between the endogenous regressor and the regression error that cannot be neglected in practice. 
	
	In this paper, we show that the empirical distribution function of nonparametric IV residuals converges weakly to a Gaussian process at a parametric rate, even though residuals are obtained from the nonparametrically estimated ill-posed inverse problem. To the best of our knowledge, this is the first result on the distribution of the nonparametric IV residuals, which can be used to develop various residual-based specification tests and is of independent interest. Building on this result, we obtain the large sample approximation to the distributions of independence separability tests. The distributions of residual-based independence tests are non-standard and not amenable to standard bootstrap approximations. Therefore, we suggest using the $m$ out of $n$ bootstrap or subsampling to compute the critical values.
	
	Our results are based on the insight that the Tikhonov regularization in Sobolev spaces, considered in \cite{florens2011identification}, \cite{gagliardini2012tikhonov}, \cite{carrasco2013asymptotic}, and \cite{gagliardini2017specification}, among others, provides a natural link between the modern empirical process theory and the theory of ill-posed inverse problems. In regards to this literature, we obtain new results for the Tikhonov regularization with a Sobolev penalty that can be applied to generic ill-posed inverse problems, including various nonparametric IV estimators, e.g., based on kernel smoothing. In particular, the Tikhonov regularization with a Sobolev penalty achieves sufficiently fast convergence rates for the semiparametric theory. In contrast, the simple one-step Tikhonov regularization without Sobolev penalization suffers from the well-known saturation effects; see \cite{darolles2011nonparametric}.
	
	The paper is organized as follows. In Section~\ref{sec:separability_test}, we present two motivating examples, where economic considerations lead to nonseparable models with endogeneity and discuss a testable implication of separability. In Section~\ref{sec:test_statistics}, we characterize the large sample approximation to the distribution of the residual-based Kolmogorov-Smirnov and Cram\'{e}r-von Mises independence tests and introduce a resampling procedure to compute the critical values. We also study the behavior of these tests under the fixed and the local alternative hypotheses.  We report on a Monte Carlo study in Section~\ref{sec:mc} which provides insights about the validity of our asymptotic approximations in finite samples. In Section~\ref{sec:empirics}, we test the separability of Engel curves for a large set of commodities and find that the separability is rejected most of the time. Conclusions appear in Section~\ref{sec:conclusions}. All technical details, auxiliary results, and proofs are collected in the Appendix and the Supplementary Material.
	
	\section{Separability of unobservables}\label{sec:separability_test}
	\subsection{Motivating examples}
	The instrumental variable models with additively separable unobservables constitute a workhorse of modern empirical practice. However, the additive separability of unobservables is a restrictive modeling assumption that essentially rules out the heterogeneity of estimated causal structural effects in unobservables; see, e.g., \cite{heckman2001micro} or \cite{imbens2007nonadditive}. Indeed, the structural economic models typically lead to nonseparable unobservables as illustrated below.
	
	\begin{example}[Demand function]
		Consider a random utility maximization problem
		\begin{equation*}
			Q = \underset{q\in\R^J:\;P^\top q = I}{\mathrm{argmax}}\;U(q,\varepsilon),
		\end{equation*}
		where $U(.,.)$ is a utility function, $Q$ is a vector of demanded quantities, $\varepsilon$ is an individual preference variable, unobserved by the econometrician, $P$ is a vector of prices, and $I$ is the income. The solution to this optimization problem leads to the nonseparable demand functions $Q_j = \Phi(P,I,\varepsilon)$ for each good $j = 1,\dots,J$ as shown in \cite{brown1989random} and \cite{lewbel2001demand}; see also \cite{hoderlein2018estimating} for the welfare analysis based on the nonseparable model. The nonseparable demand functions may lead in turn to the nonseparable Engel curves.
	\end{example}
	
	\begin{example}[Production function/frontier]
		\cite{simar2016unobserved} consider a production process with unobserved heterogeneity that leads to the production function/frontier $\phi$ such that $Y = \phi(Z,\varepsilon) - U$, where $Y$ is an output, $Z$ are observed inputs, $\varepsilon$ is an environmental factor, and $U\geq 0$ is a measure of inefficiency. In this example, the nonseparable model is generated by the fact that the environmental factor is taken into account along with other input choices by firms, and, at the same time, the former is not observed by the econometrician.
	\end{example}
	
	\subsection{A testable implication}
	Let $(Y,Z,W)$ be observed random variables admitting a nonseparable representation
	\begin{equation}\label{eq:nonseparable}
		Y = \Phi(Z,\varepsilon),\qquad \varepsilon\si W,
	\end{equation}
	where $Y\in\R$ is outcome, $Z\in\R^p$ are regressors, $\varepsilon\in\R$ is unobservable, $W\in\R^q$ is a vector of instrumental variables, and $\Phi:\R^{p}\times \R\to\R$ is a structural function. We assume that $W$ are valid instrumental variables satisfying the exclusion restriction, $\varepsilon\si W$, and the relevance condition, $W\not\si Z$. Note that the independence exclusion restriction is a commonly used identifying condition for nonseparable models; see \cite{chernozhukov2020semiparametric}, \cite{blundell2017nonparametric}, \cite{torgovitsky2017minimum}, \cite{torgovitsky2015identification}, \cite{d2015identification}, \cite{dunker2014iterative}, \cite{gagliardini2012tikhonov}, and \cite{horowitz2007nonparametric} for recent examples and applications, as well as \cite{chernozhukov2013quantile}, \cite{matzkin2013nonparametric}, and \cite{imbens2007nonadditive} for the review of earlier econometrics literature on the identification of nonseparable models.
	
	The independence condition $\varepsilon\si W$ does not rule-out the heteroskedasticity in the distribution of $Y$ conditionally on $Z$ or $W$, which is often observed in the empirical practice. It also does not rule-out the heteroskedasticity in the distribution of unobservables $\varepsilon$ conditionally on covariates $Z$. However, it rules out the heteroskedasticity of unobservables conditionally on the instrumental variable, which could be less restrictive, since the instrumental variable is univariate in typical applications. This leads to an interesting trade-off between the heterogeneity of causal structural effects in unobservables allowed for in the nonseparable model and the heteroskedasticity of unobservables conditionally on the instrumental variable allowed for in the separable model.
	
	To develop the separability test, several strategies can be adopted. For instance, one could nonparametrically estimate the nonseparable model and check whether the separability holds. This approach corresponds to the principle behind the Wald test for parametric models. Alternatively, since the nonparametric identification and estimation of the separable model is easier, one could estimate the separable model and check the independence condition of the nonseparable model. This approach corresponds to the principle behind Rao's score test in the parametric setting and is the one adopted in this paper.
	
	We say that the model in equation~(\ref{eq:nonseparable}) has a \textit{separable representation} if there exists measurable functions $\psi:\R^p\to\R$ and $g:\R\to\R$ such that
	\begin{equation*}
		Y = \psi(Z) + g(\varepsilon).
	\end{equation*}
	
	If the model has a separable representation, then the structural function can be estimated consistently using the nonparametric IV approach; see \cite{darolles2011nonparametric}, \cite{blundell2007semi}, \cite{horowitz2007nonparametric}, and \cite{newey2003instrumental}. The nonparametric IV regression function $\varphi:\R^p\to\R$ solves the functional equation
	\begin{equation}\label{eq:npiv_moment}
		r(w) \triangleq \E[Y|W=w]f_W(w) = \int\varphi(z)f_{ZW}(z,w)\dx z \triangleq (T\varphi)(w),
	\end{equation}
	where $T:L_2(\R^p)\to L_2(\R^q)$ is an integral operator. Let $U\triangleq Y - \varphi(Z)$ be the nonparametric IV regression error. Note that even if the model is nonseparable, we still have $\E[U|W]=0$ with $U=Y-\varphi(Z)$ for $\varphi$ solving the functional equation (\ref{eq:npiv_moment}). The following result provides a convenient for us testable implication of separability, provided that $U$ is unambiguously defined, see Appendix for a formal proof.
	
	\begin{proposition}\label{prop:separability}
		Suppose that there exists a unique solution to equation~(\ref{eq:npiv_moment}). If the model in equation~(\ref{eq:nonseparable}) admits a separable representation, then $U\si W$.
	\end{proposition}
	It is worth mentioning that the independence between $U$ and $W$ is only a \textit{testable implication} of additive separability of unobservables. However, when the model is nonseparable, we have $U=\Phi(Z,\varepsilon)-\varphi(Z)\triangleq h(Z,\varepsilon)$, for some non-degenerate function $h$ of $(Z,\varepsilon)$, which in many cases is not independent of $W$, because $Z\not\si W$ by the relevance condition. Therefore, the independence test between $U$ and $W$ will have power against many interesting deviations from the separability. Note also that Proposition~\ref{prop:separability} relies on the injectivity of $T$, which is known as a completeness condition, see \cite{newey2003instrumental} and \cite{babii2017completeness}, and does not require that the nonseparable model is identified; see, e.g., \cite{chernozhukov2005iv} and \cite{chen2014local}. Lastly, note that the additive separability is different from the multiplicative separability when $Y=\psi(Z)g(\varepsilon)$. However, when $Y,\psi(Z)$, and $g(\varepsilon)$ are positive, we obtain the additively separable model after taking logs.
	
	\section{Independence test}\label{sec:test_statistics}
	In this section, we introduce tests of the independence condition characterized in Proposition~\ref{prop:separability}. Formally, we focus on testing
	\begin{equation*}
		H_0:\; U\si W\qquad \text{vs.} \qquad U\not\si W.
	\end{equation*}
	$H_0$ is testable, provided that the nuisance parameter $\varphi$ in $U=Y-\varphi(Z)$ is replaced by the appropriate estimator. 
	
	\subsection{Tikhonov regularization in Sobolev spaces}
	We focus on the Tikhonov-regularized estimator penalized by the Sobolev norm to estimate the nuisance parameter $\varphi$; see \cite{carrasco2013asymptotic}, \cite{gagliardini2012tikhonov}, and \cite{florens2011identification}. The attractive feature of this estimator is that it does not suffer from the well-known saturation bias and can achieve a sufficiently fast convergence rate for our asymptotic theory and more generally for semiparametric applications; see Corollary~\ref{cor:rates_npiv} in the Appendix. 
	
	Let $(L_2(\mathbf{R}^p),\|.\|)$ denote the space of functions square-integrable with respect to the Lebesgue measure. Let $\langle x\rangle^s\triangleq (1+|x|^2)^{s/2}$ be a polynomial weight function with $s\in \R$, where $x\in\R^p$ and $|.|$ is a Euclidean norm on $\R^p$. Consider the operator $L^sf=F^{-1}(\langle .\rangle^sFf)$ defined for all $f$ such that $\|\langle.\rangle^2 Ff\|<\infty$, where $F$ is a Fourier transform on $L_2(\R^p)$ with scaling $(2\pi)^{-p/2}$. Then the self-adjoint operator $L$ generates a Hilbert scale of Sobolev spaces
	\begin{equation*}
		H^s(\R^p) = \left\{f\in L_2(\R^p):\; \|f\|_{s} \triangleq \|L^sf\| <\infty \right\};
	\end{equation*}
	see \cite{krein1966scales} for more details on Banach and Hilbert scales. 
	
	Let $(\hat T,\hat r)$ be the kernel estimators of $(T,r)$ in equation~(\ref{eq:npiv_moment}) computed as
	\begin{equation}\label{eq:r_and_T}
		\begin{aligned}
			\hat r(w) = \frac{1}{nh_n^q}\sum_{i=1}^nY_iK_w\left(h_n^{-1}(W_i - w)\right),\qquad (\hat T\phi)(w)  = \int\phi(z)\hat f_{ZW}(z,w)\dx z, \\
			\hat f_{ZW}(z,w)  = \frac{1}{nh_n^{p+q}}\sum_{i=1}^nK_z\left(h_n^{-1}(Z_i - z)\right)K_w\left(h_n^{-1}(W_i - w)\right),
		\end{aligned}
	\end{equation}
	where $K_z:\R^p\to \R$ and $K_w:\R^q\to\R$ are kernel functions and $h_n\to 0$ is a sequence of bandwidth parameters. 
	
	We estimate $\varphi$ using the Tikhonov-regularized estimator penalized by the Sobolev norm with $s\geq 0$
	\begin{equation*}
		\hat\varphi = \argmin_{\phi}\left\|\hat T \phi - \hat r\right\|^2 + \alpha_n\|\phi\|_s^2,
	\end{equation*}
	where $\hat T$ and $\hat r$ are as described above. It is easy to see that this problem has a closed-form solution
	\begin{equation*}
		\hat \varphi = L^{-s}(\alpha_n I + \hat T_s^*\hat T_s)^{-1}\hat T_s^*\hat r,
	\end{equation*}
	where $\hat T_s = \hat TL^{-s}$ and $\hat T_s^*$ is the adjoint operator to $\hat T_s$. 
	
	\subsection{Distribution of statistics}
	Let $\hat U_i=Y_i - \hat{\varphi}(Z_i)$ be the nonparametric IV residuals and let
	\begin{equation}\label{eq:edf}
		\hat F_{\hat UW}(u,w) = \frac{1}{n}\sum_{i=1}^n\one_{\{\hat U_i\leq u,W_i\leq w \}},\quad \hat F_{\hat U}(u) = \frac{1}{n}\sum_{i=1}^n\one_{\{\hat U_i\leq u\}},\quad \hat F_W(w) = \frac{1}{n}\sum_{i=1}^n\one_{\{W_i\leq w\}}
	\end{equation}
	be the empirical distribution functions. To test $H_0$, we focus on the following residual-based independence empirical process
	\begin{equation*}
		\mathbb{G}_n(u,w) = \sqrt{n}\left(\hat F_{\hat UW}(u,w) - \hat F_{\hat U}(u)\hat F_W(w)\right).
	\end{equation*}
	Note that this process involves residuals $\hat U_i$ instead of the true regression errors $U_i$, hence, its asymptotic behavior can be significantly different from the asymptotic behavior of classical independence empirical processes; see \cite{van2000weak}, Chapter 3.8. In particular, the estimation of the nuisance component $\varphi$ may affect the asymptotic distribution of the independence empirical process.
	
	To understand the behavior of $\mathbb{G}_n$, we introduce several assumptions.
	\begin{assumption}\label{as:hilbert_scale}
		For some $a,b>0$
		\begin{enumerate}
			\item[(i)] Operator smoothing: $\|T\phi\|_v\sim \|\phi\|_{v-a}$ for all $\phi\in L_2(\R^p)$ and $v\in\R$.
			\item[(ii)] Parameter smoothness: $\varphi\in H^b(\R^p)$.
		\end{enumerate}
	\end{assumption}
	Assumption~\ref{as:hilbert_scale} (i) describes the smoothing property of the operator $T$. Roughly speaking, the action of $T$ increases the Sobolev smoothness by $a$, which is called the degree of ill-posedness. Intuitively, the more $T$ smooths out features of $\varphi$, the harder it is to recover it from the equation~(\ref{eq:npiv_moment}). Condition (ii) describes the smoothness of the structural function $\varphi$ and is a standard smoothness restriction in the nonparametric literature.
	
	\begin{assumption}\label{as:dgp}
		(i) $(Y_i,Z_i,W_i)_{i=1}^n$ are i.i.d. observations of $(Y,Z,W)$ with $\E|Y|^2<\infty$ $\E\|W\|<\infty$, $\E\|Z\|<\infty$, and $\E\left[U^2|W\right]\leq C<\infty$; (ii) the distribution of $(U,Z,W)$ is absolutely continuous with respect to the Lebesgue measure with densities $f_Z,f_W,f_{ZW},f_{U|Z}\in L_\infty$ and $f_Z,f_{ZW}\in L_2$; (iii) $f_{ZW}\in H^t(\R^{p+q})$ for some $t>0$; (iv) $K_z$ and $K_w$ products of a univariate continuous kernel $K\in L_2(\R)\cap L_\infty(\R)$ of bounded variation with $\int K(u)\dx u = 1$, $\int|u|^l|K(u)|\dx u<\infty$, and $\int u^kK(u)\dx u=0$ for $k\in\{1,\dots,l\}$ and $l\geq t$.
	\end{assumption}
	Assumption~\ref{as:dgp} describes several mild conditions on the distribution of the data and the kernel functions that are largely standard for kernel estimators; see also \cite{darolles2011nonparametric}, Appendix B for a discussion of generalized boundary kernels that can be used when supports are bounded. To introduce the next assumption, let $\partial_u$ be a partial derivative with respect to the variable $u$, let $\|.\|_\infty$ denote the uniform norm, and put $x\vee y = \max\{x,y\}$ and $x\wedge y = \min\{x,y\}$. 
	
	\begin{assumption}\label{as:smoothness}
		(i) $\|\partial_u f_{UZ}\|_{\infty} < \infty$ and $\sup_u\|f_{UZ}(u,.)\|_{\kappa} < \infty$ with $\kappa> 2a\vee(a+q/2)$; (ii) $\left\|\int_{\{v\leq .\}}\partial_u f_{UZW}(.,.v)\dx v\right\|_\infty < \infty$ and $\sup_{u,w}\left\|\int^wf_{UZW}(u,.,v)\dx v\right\|_{\kappa}<\infty$ with $\kappa>2a\vee(a+q/2)$;
	\end{assumption}
	Assumption~\ref{as:smoothness} imposes some relatively mild smoothness conditions on the distribution of the data.
	
	\begin{assumption}\label{as:tuning}
		$h_n\to 0$ and $\alpha_n\to 0$ as $n\to\infty$ are such that (i) $nh_n^q\alpha_n^{2(a+c)/(a+b)}\to\infty$, $nh_n^{p+q}\alpha_n\to\infty$, and $h_n^{2t}/\alpha_n\to 0$; (ii) $\sqrt{n}\alpha_n^{2b/(a+b)}\to 0$, $\sqrt{n}h_n^q\alpha^{2a/(a+b)}\to\infty$, and $\sqrt{n}h_n^{2t}/\alpha_n^{2a/(a+b)}\to 0$; (iii) $n\alpha_n^2\to0$, $nh_n^{2(b\wedge2t)}\to 0$, and $nh_n^{p+2q}\to\infty$; where $2s=b-a\geq 0$, $b>c$, $s\geq c>p/2$, $t>(p+q)/2$,  $a,b,t,p,q$ are as in Assumptions~\ref{as:hilbert_scale} and \ref{as:dgp}.
	\end{assumption}
	Assumption~\ref{as:smoothness} (i) provides a set of sufficient conditions for $\|\hat\varphi - \varphi\|_c = o_P(1)$ with $c>p/2$, while condition (ii) states additional requirements for $\|\hat\varphi - \varphi\| = o_P(n^{-1/4})$; see Corollary~\ref{cor:rates_npiv} in the Appendix. The former condition is needed for the asymptotic equicontinuity argument, while the latter requires that the nuisance parameter $\varphi$ is estimated at a sufficiently fast rate, which is often encountered in the semiparametric literature. Lastly, condition (iii) ensures that a certain uniform asymptotic expansion holds. To illustrate that conditions on tuning parameters are feasible, suppose for simplicity that $p=q=c=1$ and that $h_n\sim n^{-c_1}$ and $\alpha_n\sim n^{-c_2}$ for some $c_1,c_2\in(0,1)$. Then (i) requires  that $c_1+2c_2(a+c)/(a+b)<1$, $2c_1+c_2<1$, and $t>c_2/2c_1$. For (ii), we additionally need $c_2>(a+b)/4b$, $c_1+2c_2a/(a+b)<0.5$, and $t>1/4c_1+c_2a/c_1(a+b)$. Lastly, (iii) requires that $c_2>0.5$, $c_1>1/2(b\wedge 2t)$, and $c_1<1/3$. Therefore, we require $(c_1,c_2)\in\{(x,y)\in\R^2: 0.5<y<1-2x,x\in(1/2(b\wedge 2t),1/3) \}$, which is non-empty provided that $b\wedge 2t>3/2$. Given this choice, the following smoothness conditions are imposed in Assumption~\ref{as:tuning}: $t>[1/4c_1+c_2a/c_1(a+b)]\wedge [c_2/2c_1]$ and $b>[2c_2(a+c)/(1-c_1) - a]\wedge[2c_2a/(0.5-c_1) - a]\wedge [1/(c_2-0.25)]$.
	
	The following result describes a convenient for us approximation to the residual-based independence empirical process:
	\begin{theorem}\label{thm:main}
		Suppose that Assumptions~\ref{as:hilbert_scale}, \ref{as:dgp}, \ref{as:smoothness}, and \ref{as:tuning} are satisfied. Then
		\begin{equation*}\footnotesize
			\begin{aligned}
				\mathbb{G}_n(u,w) =  \frac{1}{\sqrt{n}}\sum_{i=1}^n\one_{\{U_i\leq u,W_i\leq w\}} - \one_{\{U_i\leq u\}}F_W(w) - \one_{\{W_i\leq w\}}F_U(u) + F_{UW}(u,w) + \delta_{u,w}(U_i,W_i) + o_P(1)
			\end{aligned}
		\end{equation*}
		uniformly over $(u,w)\in\R\times\R^q$ with
		\begin{equation*}
			\begin{aligned}
				\delta_{u,w}(U_i,W_i) & = U_i\left(T(T^*T)^{-1}\rho(u,.,w)\right)(W_i), \\
				\rho(u,z,w) & = \int^wf_{UZW}(u,z,\tilde w)\dx \tilde w - f_{UZ}(u,z)F_W(w).
			\end{aligned}
		\end{equation*}
	\end{theorem}
	It if worth mentioning that Theorem~\ref{thm:main} does not require $U\si W$. The proof of this result can be found in the Appendix and relies on the asympttoic equicontinuity arguments. Roughly speaking, we show that the consistency of the nonparametric IV estimator in the Sobolev norm together with the Donsker property of Sobolev balls imply that that certain terms associated with residuals are asymptotically negligible. At the same time, the estimation of the nuisance component $\varphi$ has a first-order asymptotic effect due to the $\delta_{u,w}(U_i,W_i)$ term, while the higher-order terms are negligible provided that $\|\hat\varphi - \varphi\|=o_P(n^{-1/4})$. This rate condition is typically encountered for the semiparametric problems; see \cite{chernozhukov2018double} and \cite{chernozhukov2016locally} for recent contributions, \cite{andrews1994empirical} for earlier treatment, and \cite{babii2020high}, Section 3.3 for a related discussion in the setting of ill-posed inverse problems.
	
	It is worth mentioning that, in some cases, the estimation of nuisance parameters does not have any first-order asymptotic effect, which is known as the Neyman orthogonality property in the semiparametric literature. In particular, this is the case for the independence empirical process based on the nonparametric conditional mean regression residuals; see  \cite{einmahl2008specification}. Interestingly, if we had $W\si(U,Z)$, then $\rho=0$, and the estimation of $\varphi$ would not have any first-order asymptotic effect.
	
	Theorem~\ref{thm:main} can be readily used to construct the residual-based Cram\'{e}r-von Mises and Kolmogorov-Smirnov statistics
	\begin{equation*}
		\begin{aligned}
			T_{2,n} = \iint |\mathbb{G}_n(u,w)|^2\dx \hat F_{\hat UW}(u,w) \qquad \text{and}\qquad T_{\infty,n} = \sup_{u,w}|\mathbb{G}_n(u,w)|.
		\end{aligned}
	\end{equation*}
	To understand the behavior of the two statistics under the null and the alternative hypotheses, consider a centered version of the process in Theorem~\ref{thm:main}
	\begin{equation*}
		\mathbb{H}_n(u,w) = \frac{1}{\sqrt{n}}\sum_{i=1}^nh_{u,w}(U_i,W_i) - \E[h_{u,w}(U_i,W_i)] ,
	\end{equation*}
	where $h_{u,w}(U,W) = \one_{\{U\leq u,W\leq w\}} - \one_{\{U\leq u\}}F_W(w) - \one_{\{W\leq w\}}F_U(u) + F_{UW}(u,w) + \delta_{u,w}(U,W)$. The following Donsker-type central limit theorem holds:
	\begin{proposition}\label{corr:main}
		Suppose that assumptions of Theorem~\ref{thm:main} are satisfied. Then
		\begin{equation*}
			\mathbb{H}_n \leadsto \mathbb{H}\qquad \mathrm{in}\qquad L_\infty(\R\times\R^q),
		\end{equation*}
		where $\mathbb{H}$ is a tight centered Gaussian process with uniformly continuous sample paths and the covariance function
		\begin{equation*}
			\begin{aligned}
				(u,w,u',w') \mapsto \E\left[(h_{u,w}(U,W) - \E[h_{u,w}(U,W)])(h_{u',w'}(U,W) - \E[h_{u',w'}(U,W)])\right].
			\end{aligned}
		\end{equation*}
	\end{proposition}
	Note that under the null hypothesis $H_0:U\si W$, we have $\E[h_{u,w}(U,W)]=0$ and the covariance function of $\mathbb{H}$ simplifies to
	\begin{equation*}\footnotesize
		\begin{aligned}
			(u,w,u',w') \mapsto & \E\left[\left(\one_{\{U\leq u,W\leq w\}} - \one_{\{U\leq u\}}F_W(w) - \one_{\{W\leq w\}}F_U(u) + F_{UW}(u,w) + \delta_{u,w}(U,W)\right)\times\right. \\
			& \quad\left.\times\left(\one_{\{U\leq u',W\leq w'\}} - \one_{\{U\leq u'\}}F_W(w') - \one_{\{W\leq w'\}}F_U(u') + F_{UW}(u',w') + \delta_{u',w'}(U,W)\right)\right].
		\end{aligned}
	\end{equation*}
	For the alternative hypothesis, $H_1:U\not\si W$, put
	\begin{equation*}
		\begin{aligned}
			d_2 & = \iint|F_{UW}(u,w) - F_U(u)F_W(w)|^2\dx F_{UW}(u,w),\quad d_\infty = \sup_{u,w}|F_{UW}(u,w)-F_U(u)F_W(w)|. \\
		\end{aligned}
	\end{equation*}
	Consider also a sequence of local alternative hypotheses
	\begin{equation*}
		H_{1,n}:\; F_{UW}(u,w)=F_U(u)F_W(w)+ n^{-1/2}H(u,w),\qquad \forall u,w,
	\end{equation*}
	where the function $H$ is such that $F_{UW}$ is a proper CDF. There exist several ways to construct such local alternatives with prespecified marginal distributions $F_U$ and $F_W$. For instance, the Morgenstern's family is $F_{UW}(u,w)=F_{U}(u)F_{W}(w) + aF_{U}(u)F_{W}(w)(1-F_U(u))(1-F_W(w))$ with $a\in[-1,1]$; see \cite{devroye1986nonuniform}, Chapter XI, Theorem 3.2. The following corollary describes the behavior of the independence test under the null and fixed/local alternative hypotheses:
	
	\begin{corollary}\label{cor:asympt_distributions}
		Suppose that assumptions of Theorem~\ref{thm:main} are satisfied. Then under $H_0$
		\begin{equation*}
			T_{2,n} \leadsto\iint |\mathbb{H}(u,w)|^2\dx F_{UW}(u,w)\qquad \text{and}\qquad T_{\infty,n} \leadsto\sup_{u,w}|\mathbb{H}(u,w)|,
		\end{equation*}
		while under $H_1$, we have $T_{2,n},T_{\infty,n}\xrightarrow{\mathrm{a.s.}}\infty$, provided that $d_2,d_\infty>0$. Moreover, under $H_{1,n}$
		\begin{equation*}
			T_{2,n} \leadsto\iint |\mathbb{H}(u,w) + 2H(u,w)|^2\dx F_{UW}(u,w)\qquad \text{and}\qquad T_{\infty,n} \leadsto\sup_{u,w}|\mathbb{H}(u,w) + 2H(u,w)|.
		\end{equation*}
	\end{corollary}
	
	Corollary~\ref{cor:asympt_distributions} shows that the residual-based independence test can detect parametric local alternatives. The asymptotic distributions under $H_0$ are not pivotal, in contrast to the nonparametric regression without endogeneity, cf. \cite{einmahl2008specification}. While obtaining the distribution-free statistics is possible in simpler residual-based testing problems, see \cite{escanciano2018asymptotic}, these methods do not seem to extend naturally to our setting. Therefore, the bootstrap could be an attractive alternative for simulating the critical values of the test. Interestingly, the naive nonparametric and the multiplier bootstraps do not work.
	
	\subsection{Critical values}
	The asymptotic distributions in Corollary~\ref{cor:asympt_distributions} are nonstandard and depend on several nuisance nonparametric components. This calls for resampling methods to compute the critical values.  As can be seen from the proof of Theorem~\ref{thm:main}, our uniform asymptotic expansion relies on the differentiability of the CDF. This leads to a dependence of the asymptotic distribution on the probability density function $f_{UZW}$ in Corollary~\ref{cor:asympt_distributions}; see also the proof of Theorem~\ref{thm:residuals} and Corollary~\ref{corr:clt}. Such uniform asymptotic expansion cannot be obtained in the same way for the bootstrapped statistics since in the bootstrap world the empirical distribution function is not differentiable. 
	
	The lack of smoothness of the empirical distribution function suggests that the standard bootstrap procedures may fail in approximating the asymptotic distribution of the test statistics. The problem of a similar nature occurs with the bootstrap of the cube-root consistent estimators; see, e.g., \cite{babii2019isotonic} and references therein. Another complication with the bootstrap is that we typically need to resample from the distribution obeying the constraints of the null hypothesis and that the validity of the bootstrap has to be established case-by-case. Note also that the (smoothed) residual bootstrap, cf. \cite{neumeyer2019bootstrap}, does not preserve the dependence between the endogenous regressor and the unobservables and does not mimic the data generating process of the IV regression under the null hypothesis. In Section~\ref{sec:mc}, we find in Monte Carlo experiments that the standard nonparametric bootstrap does not work.
	
	Consequently, we suggest relying on the subsampling or the $m$ out of $n$ bootstrap to compute the critical values of the test. The resampling procedure is as follows:
	\begin{enumerate}
		\item Draw a sample of size $m$ from $(Y_i,Z_i,W_i)_{i=1}^n$ without replacements (subsampling) or with replacements ($m$ out of $n$ bootstrap), where $m=m_n$ is a sequence such that $m_n\to\infty$ and $m_n/n\to 0$ and as $n\to\infty$.
		\item Compute the Kolmogorov-Smirnov or the Cram\'{e}r-von Mises statistics using the simulated sample.
		\item Repeat the first two steps many times and compute the critical values using empirical quantiles of the statistics over all simulated samples. Alternatively, compute p-values as $1 - F_n^*(T_n)$, where $T_n$ is the statistics computed from $(Y_i,Z_i,W_i)_{i=1}^n$ and $F_n^*$ is the empirical distribution function of bootstrapped statistics.
	\end{enumerate}
	An attractive feature of subsampling is that it is valid for general hypothesis testing problems; see \cite{politis2001asymptotic}, Theorem 3.1, and there is no need to show its validity in each specific application. An adaptive data-driven rule to select $m_n$ is considered, e.g., in \cite{bickel2008choice}.
	
	\section{Monte Carlo experiments}\label{sec:mc}
	To evaluate the finite-sample performance of the test, we simulate samples as
	\begin{equation*}
		Y = \varphi(Z) + \theta ZU + U,\qquad 
		\begin{pmatrix}
			Z \\
			W \\
			U 
		\end{pmatrix} \sim_{i.i.d.} N\left(\begin{pmatrix}
			0 \\
			0 \\
			0
		\end{pmatrix}, \begin{pmatrix}
			1 & 0.4 & 0.3 \\
			0.4 & 1 & 0 \\
			0.3 & 0 & 1
		\end{pmatrix}\right).
	\end{equation*}
	We set $\varphi(x) = \cos(x)$ and consider samples of size $n=500$ and $n=1,000$ observations; see Supplementary Material for additional simulation results. Note that the degree of separability of unobservables is governed by $\theta\in\R$. The separable model corresponds to $\theta = 0$, while any $\theta\ne0$ corresponds to the alternative nonseparable model. It is worth mentioning that under $H_1$, the nonparametric IV regression does not estimate consistently the nonseparable structural function $(z,u)\mapsto \cos(z)+\theta zu$, which depends on unobservables. The nonparametric IV regression estimates instead the function $z\mapsto\phi(z)$ solving the functional equation $\E[Y|W]=\E[\phi(Z)|W]$. The difference between the two functions is precisely what gives the power to the test.
	
	We set the number of Monte Carlo replications and the number of bootstrap replications to $1,000$ through all our experiments. We also discretize all continuous quantities on the grid of 100 equidistant points in $[-4,4]$. The estimates $\hat r$ and $\hat T$ in equation~(\ref{eq:r_and_T}) are obtained using the sixth-order Epanechnikov kernel. The corresponding bandwidth parameters are computed using Silverman's rule of thumb: $h_z = 3.53\hat\sigma_zn^{-1/13}$ and $h_w = 3.53\hat\sigma_wn^{-1/13}$, where $\hat\sigma_z$ and $\hat\sigma_w$ are sample standard deviations of observed $Z$ and $W$. This choice satisfies Assumption~\ref{as:tuning} and requires that the regularization parameter is $\alpha_n\sim n^{-c_2}$ with $c_2\in(0.5, 11/13)$. To satisfy this requirement, we set $\alpha_n = n^{-4/5}$. 
	
	We look at the distributions of Kolmogorov-Smirnov and Cram\'{e}r-von Mises statistics, computed respectively as
	\begin{equation*}
		T_{\infty,n} = \sup_{u,w}\left|\mathbb{G}_n(u,w)\right|\qquad \text{and}\qquad T_{2,n} = \iint|\mathbb{G}_n(u,w)|^2\dx \hat F_{\hat UW}(u,w),
	\end{equation*}
	where $\mathbb{G}_n(u,w) = \sqrt{n}(\hat F_{\hat UW}(u,w) - \hat F_{\hat U}(u)\hat F_W(w))$ and the empirical distribution functions are computed as in equation~(\ref{eq:edf}). Lastly, we use the adaptive rule of \cite{bickel2008choice} to estimate the size of the subsample. The rule consists of choosing $\hat m_j = \argmin_{j\geq 0}\sup_x|F_j^*(x) - F_{j+1}^*(x)|$, where $F_j^*$ is the empirical distribution of the simulated statistics using a subsample of size $m_j = [q^{j}n],j=0,1,2,\dots,4$, $[a]$ is integer part of $a$, and $q=0.5$.
	\begin{figure}
		\centering
		\begin{subfigure}[b]{0.49\textwidth}
			\includegraphics[width=\textwidth]{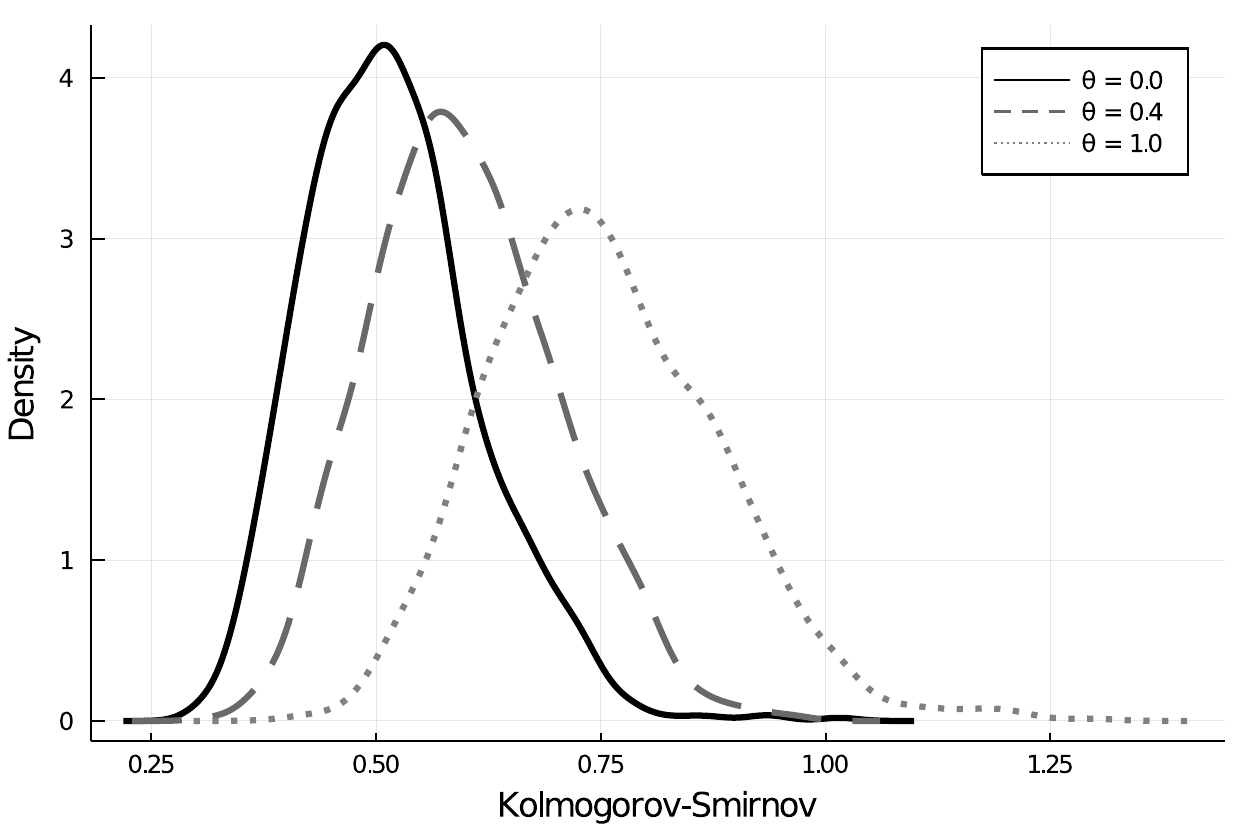}
			\caption{Sample size: $n=500$}
		\end{subfigure}
		\begin{subfigure}[b]{0.49\textwidth}
			\includegraphics[width=\textwidth]{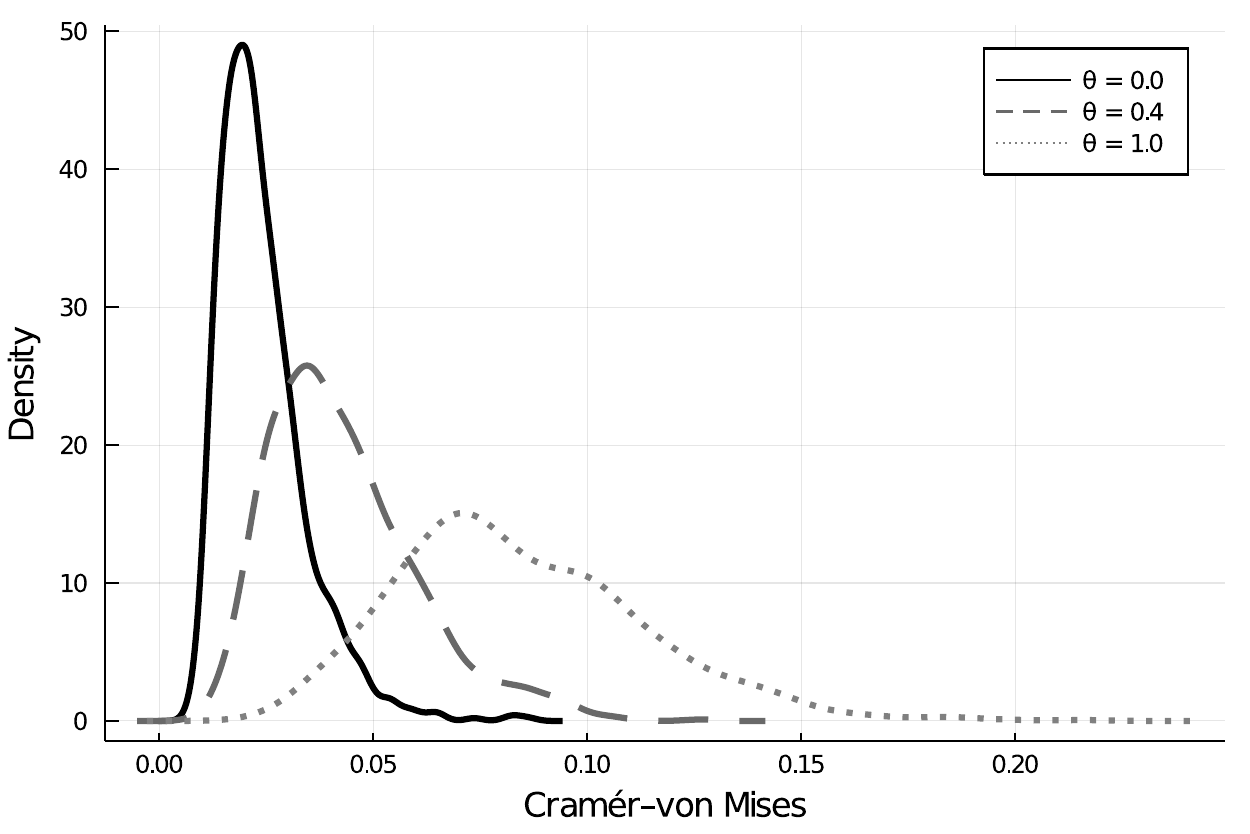}
			\caption{Sample size: $n=500$}
		\end{subfigure}
		\begin{subfigure}[b]{0.49\textwidth}
			\includegraphics[width=\textwidth]{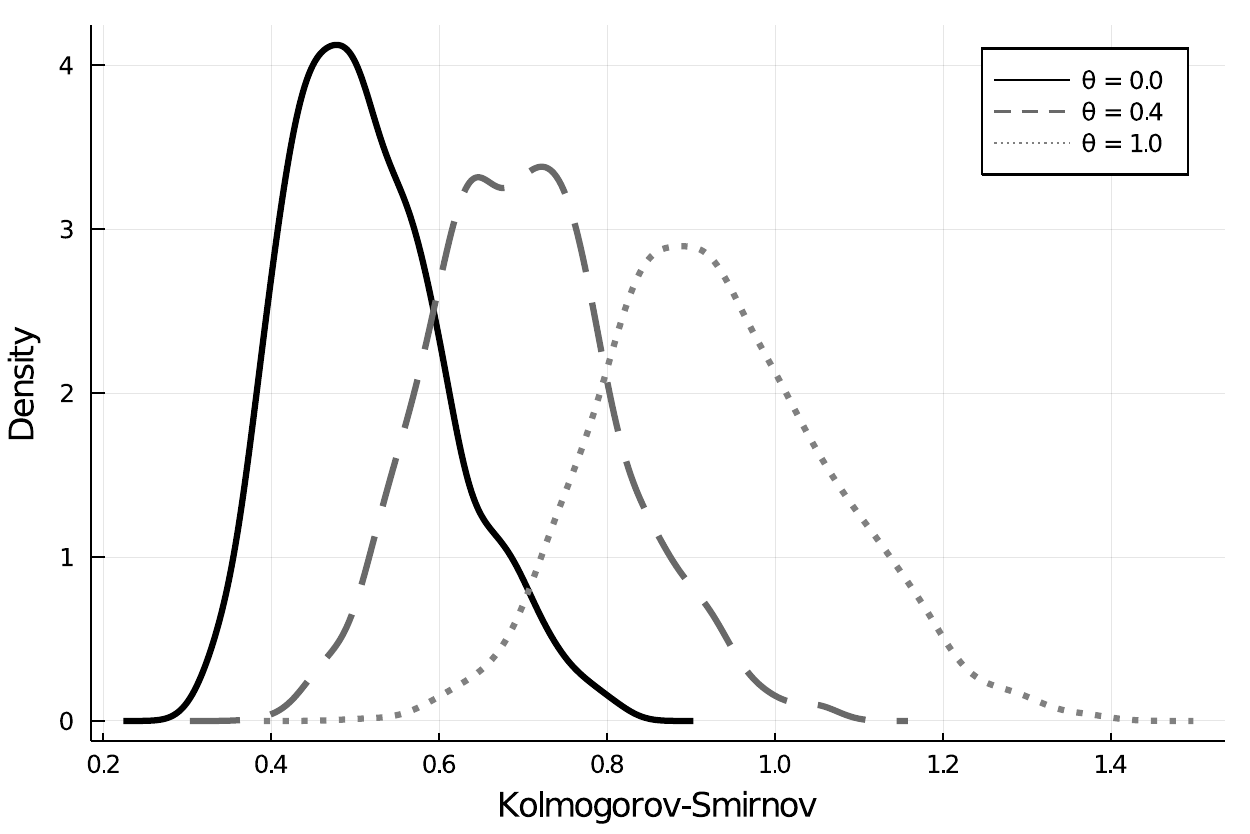}
			\caption{Sample size: $n=1,000$}
		\end{subfigure}
		\begin{subfigure}[b]{0.49\textwidth}
			\includegraphics[width=\textwidth]{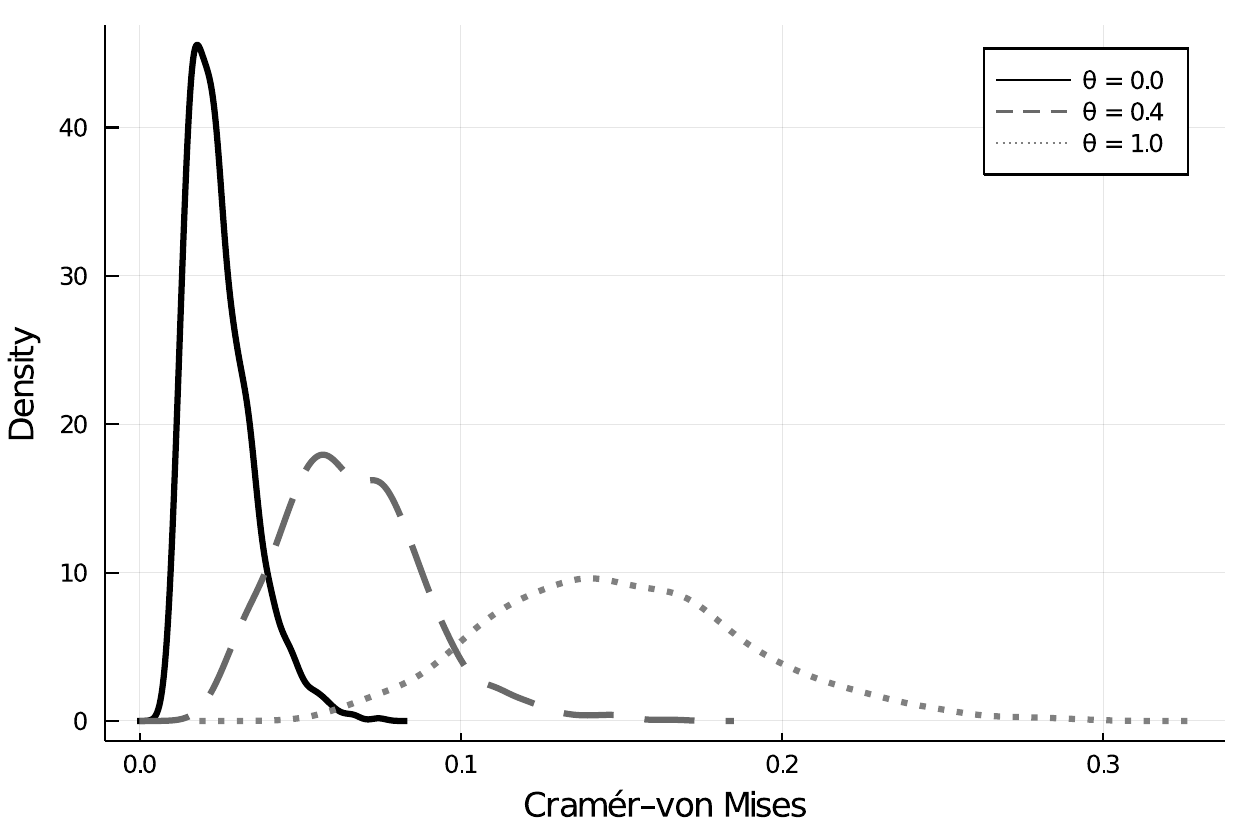}
			\caption{Sample size: $n=1,000$}
		\end{subfigure}
		\caption{Finite-sample distribution of the test. The figure shows density estimates for the Kolmogorov-Smirnov and Cram\'{e}r-von Mises statistics under $H_0$, $\theta=0$ (solid line), and two alternative hypotheses: $\theta=0.4$ (dashed line) and $\theta=1$ (dotted line).}
		\label{fig:1}
	\end{figure}
	
	Figure~\ref{fig:1} shows the distribution of the test statistics under the null hypothesis and the two alternative hypotheses for different sample sizes. The two distributions are sufficiently distinct once the alternative hypothesis becomes more separated from the null hypothesis.
	
	We plot in Figure~\ref{fig:2} the power curves when the level of the test is fixed at $5\%$. The power of the test increases once alternative hypotheses become more distant from the null hypothesis. The Cram\'{e}r-von Mises test seems to have higher power for the class of considered alternatives. We can also see that the figure illustrates the consistency of the test in the sense that its power becomes closer to one as the sample size increases under the alternative hypotheses.
	\begin{figure}
		\centering
		\begin{subfigure}[b]{0.49\textwidth}
			\includegraphics[width=\textwidth]{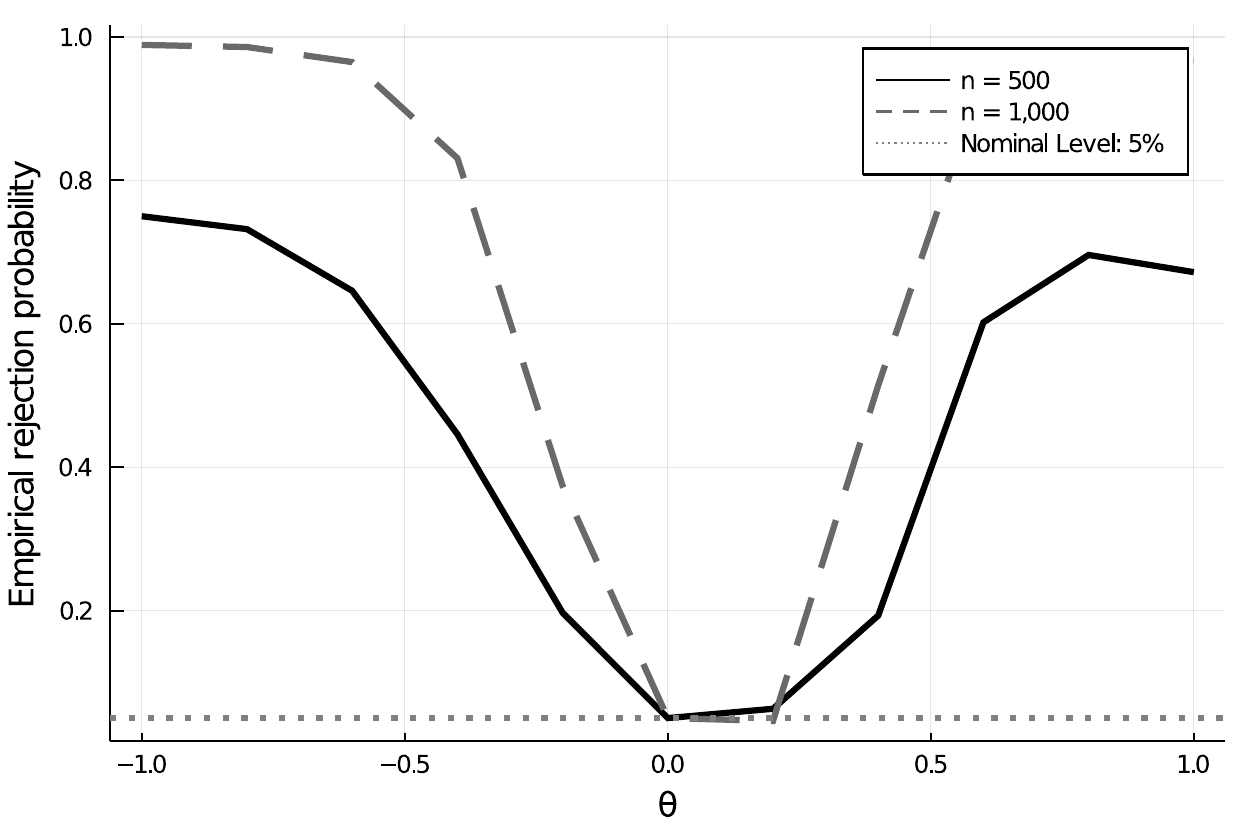}
			\caption{Kolmogorov-Smirnov test}
		\end{subfigure}
		\begin{subfigure}[b]{0.49\textwidth}
			\includegraphics[width=\textwidth]{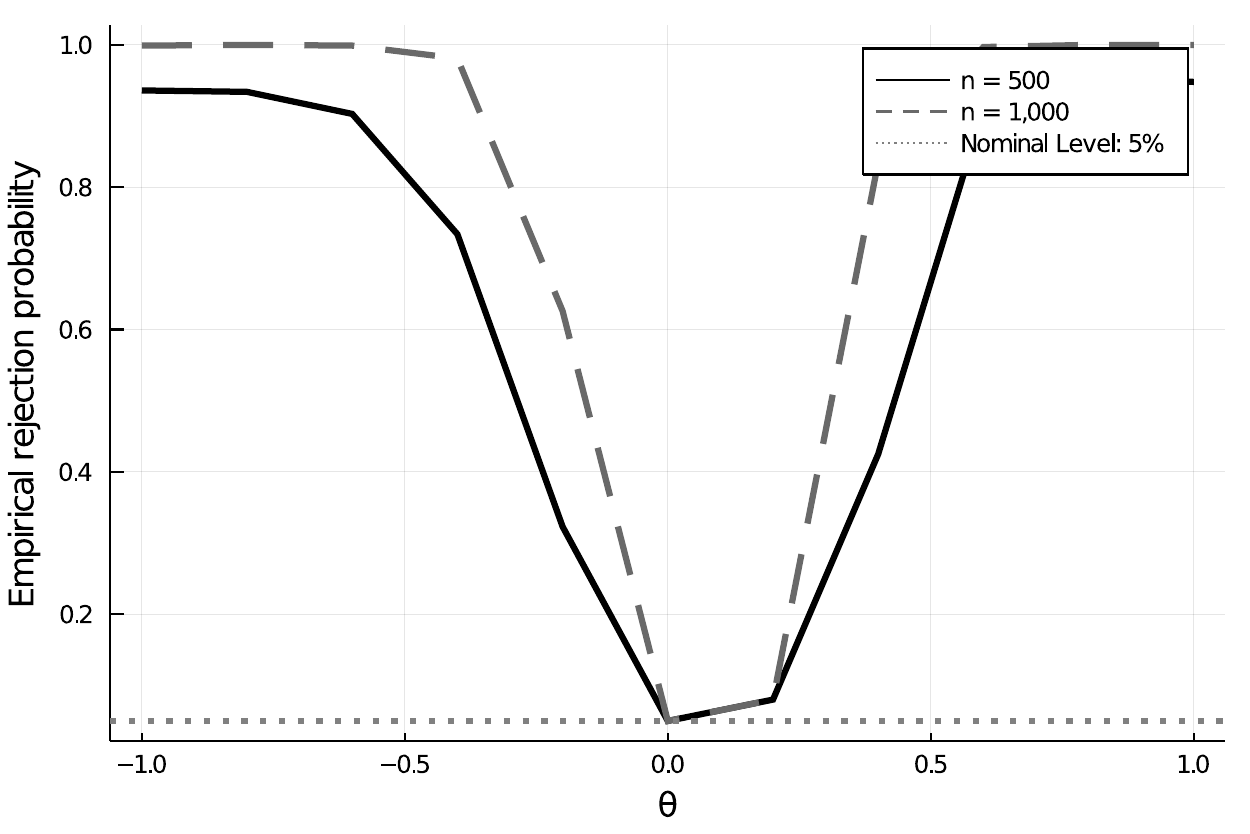}
			\caption{Cram\'{e}r-von Mises test}
		\end{subfigure}
		\caption{Power curves. The figure shows empirical rejection probabilities as a function of degree of separability $\theta$ for samples of size $n=500$ (solid line) and $n=1,000$ (dashed line). The value $\theta=0$ corresponds to the separable model, while $\theta\ne 0$ are deviations from separability. The nominal level of the test is set at $5\%$.}
		\label{fig:2}
	\end{figure}
	
	In Figure~\ref{fig:bootstrap_1000}, we explore the performance of the bootstrap. We plot the exact finite sample distribution of both test statistics and the distribution of bootstrapped statistics under $H_0$. In panels (a) and (b), we plot the distribution of the naive bootstrap, drawing a sample of size $n$ randomly with replacements from $(Y_i,Z_i,W_i)_{i=1}^n$. In panels (c) and (d), we plot the distribution of the $m$ out of $n$ bootstrap. The naive bootstrap fails and does not mimic the distribution of the Kolmogorov-Smirnov/Cram\'{e}r-von Mises statistics. The distribution of the $m$ out of $n$ bootstrap, on the other hand, is close to the finite sample distributions of both statistics. We also observe that the adaptive choice seems to work slightly better for the Cram\'{er}-von Mises statistics.
	\begin{figure}
		\centering
		\begin{subfigure}[b]{0.45\textwidth}
			\includegraphics[width=\textwidth]{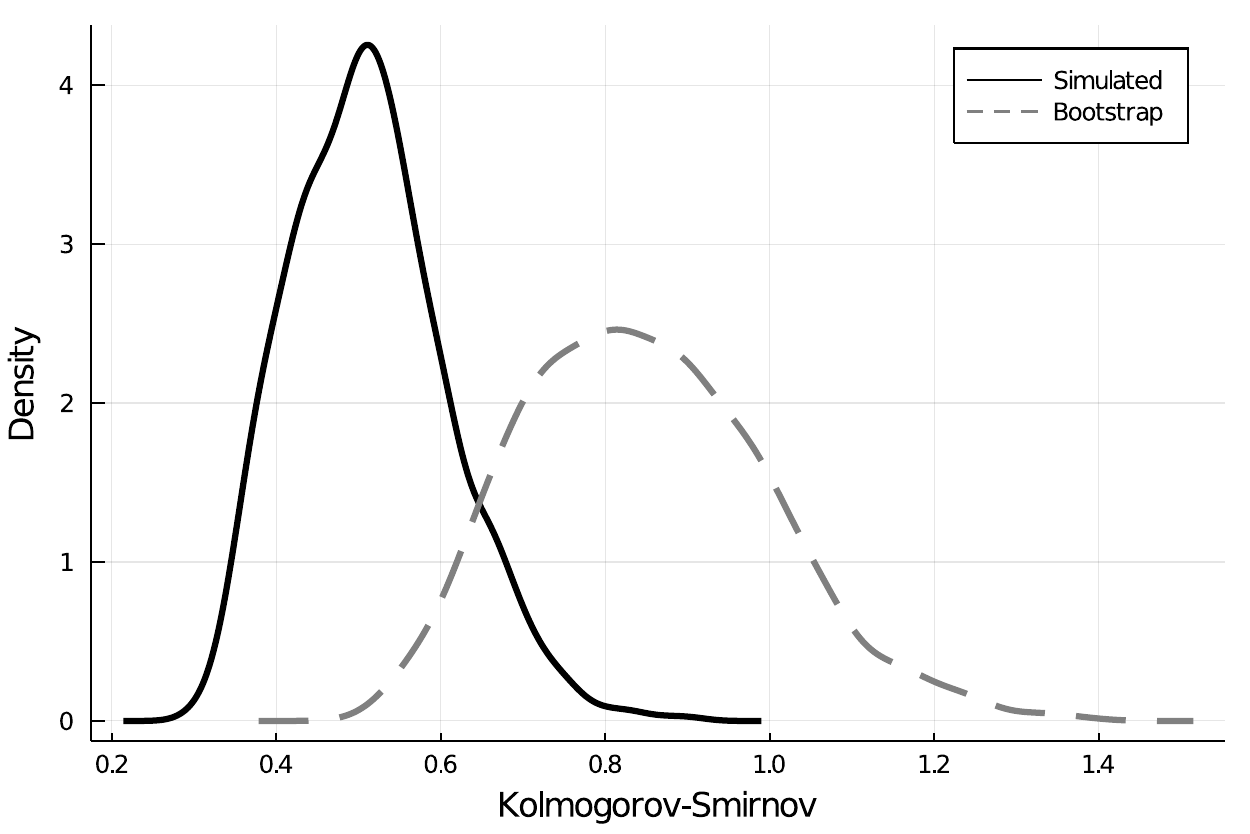}
			\caption{Naive bootstrap: $m_n=n$}
		\end{subfigure}
		\begin{subfigure}[b]{0.45\textwidth}
			\includegraphics[width=\textwidth]{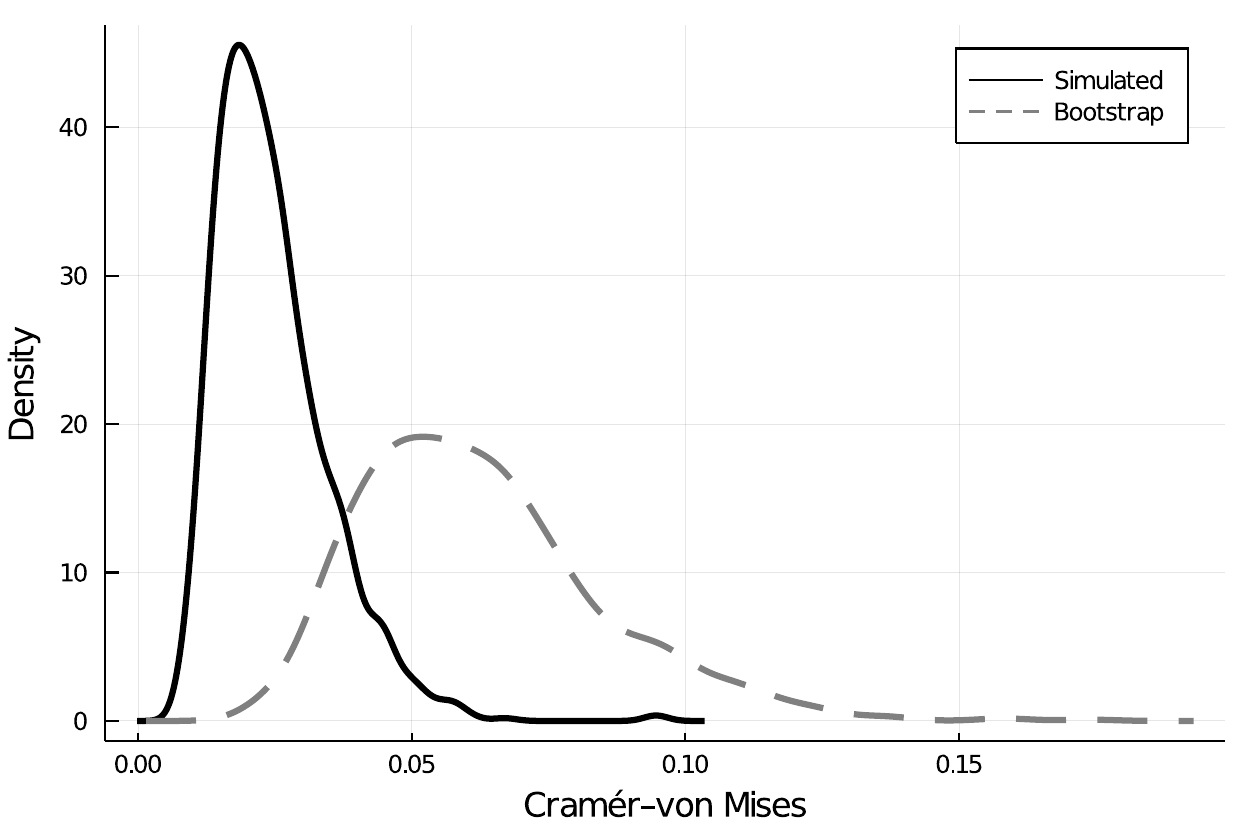}
			\caption{Naive bootstrap: $m_n=n$}
		\end{subfigure}
		\begin{subfigure}[b]{0.45\textwidth}
			\includegraphics[width=\textwidth]{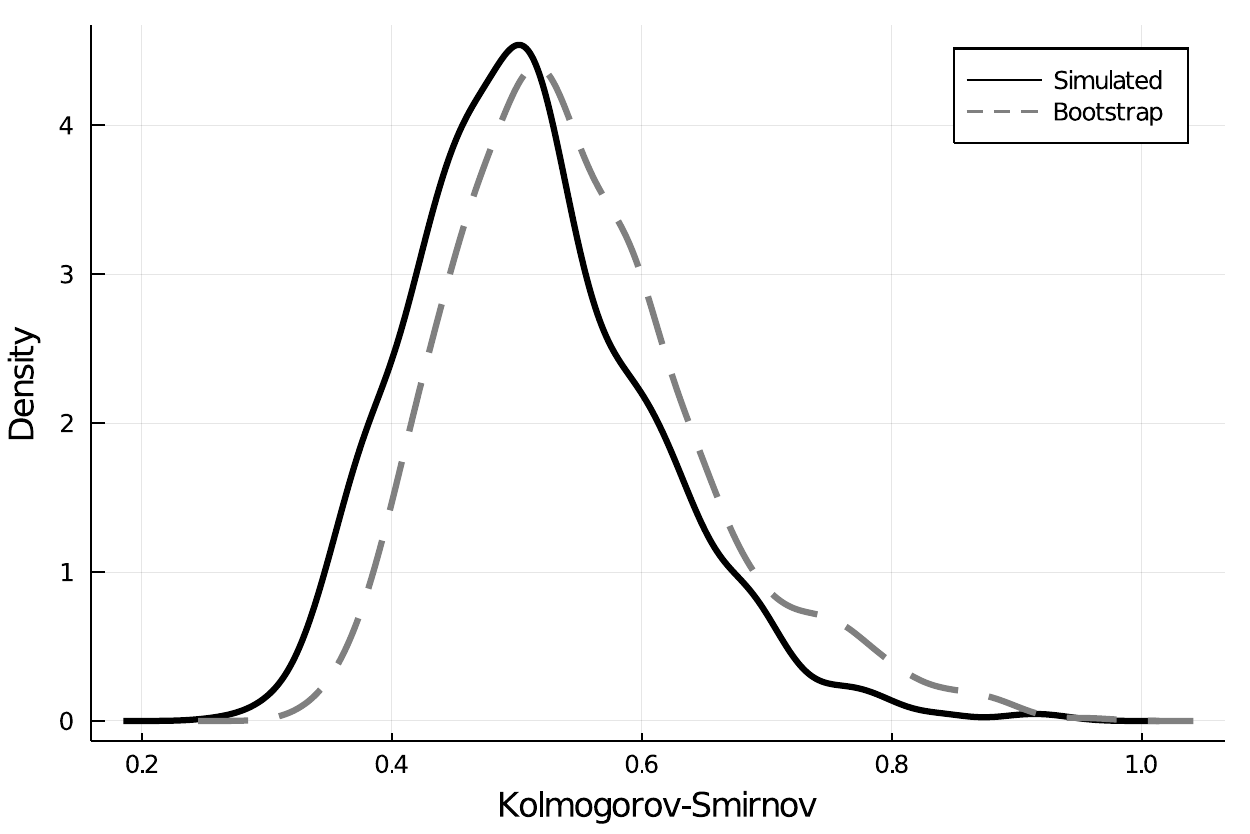}
			\caption{Adaptive choice of $m_n$}
		\end{subfigure}
		\begin{subfigure}[b]{0.45\textwidth}
			\includegraphics[width=\textwidth]{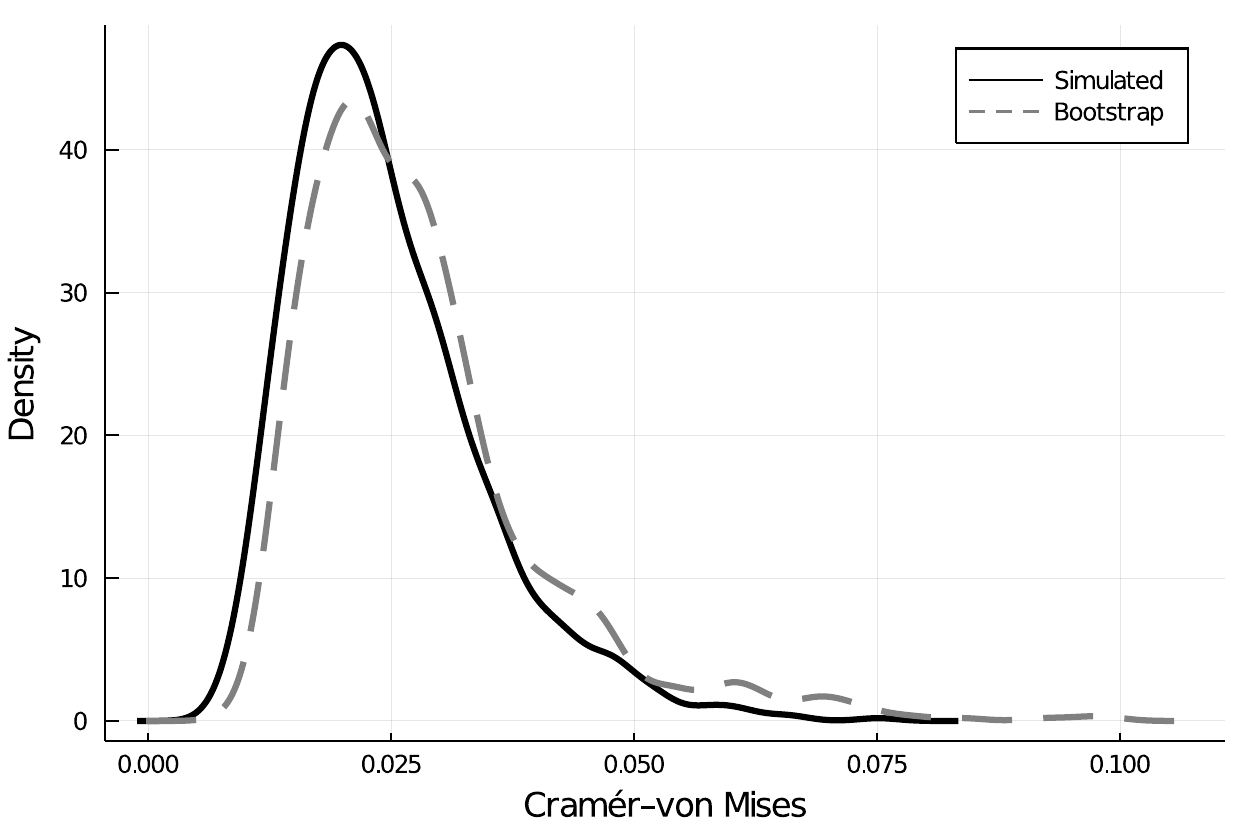}
			\caption{Adaptive choice of $m_n$}
		\end{subfigure}
		\caption{Naive bootstrap vs. $m$ out of $n$ bootstrap. The figure shows density estimates of simulated and bootstrapped Kolmogorov-Smirnov and Cram\'{e}r-von Mises statistics under $H_0$. We compare the naive ($n$ out of $n$) bootstrap and the $m$ out of $n$ bootstrap with the adaptive choice of $m$ for the sample of size $1,000$ observations.}
		\label{fig:bootstrap_1000}
	\end{figure}
	
	\section{Are Engel curves separable?}\label{sec:empirics}
	Engel curves are fundamental for the analysis of consumers' behavior and have implications for the aggregate economic outcomes. The Engel curve describes the relationship between the demand for a particular commodity and the household's budget. Interesting applications of the estimated Engel curves include a measurement of welfare losses associated with tax distortions in \cite{banks1997quadratic}, an estimation of the growth and the inflation in \cite{nakamura2014chinese}, or an estimation of the income inequality across countries in \cite{almaas2012international}. The nonparametric IV approach to the estimation of Engel curves is pioneered in the seminal paper \cite{blundell2007semi} who focus on the estimation of Engel curves in the UK.
	
	We draw a dataset from the 2015 US Consumer Expenditure Survey; see \cite{babii2016a} for the estimated Engel curves with the uniform confidence bands using this dataset. We restrict our attention to married couples with a positive income during the last 12 months, yielding 10,055 observations. The dependent variable is a share of expenditures on a particular commodity while the endogenous regressor is a natural logarithm of the total expenditures. We instrument the expenditures using the gross income. In particular, \cite{blundell2007semi} point out that the gross income will be exogenous for consumption expenditures assuming that heterogeneity in earnings is not related to unobserved preferences over consumption; see also \cite{chen2018optimal} and \cite{babii2016a}.
	
	In Table~\ref{tab:engel}, we report the $m$ out of $n$ bootstrap p-values, with the adaptive choice of $m$; see Section~\ref{sec:mc} for more details on the practical implementation of tests. We report results for both the Kolmogorov-Smirnov (KS) and the Cram\'{e}r-von Mises (CvM) tests. Remarkably, the 5\% level tests reject the separability for all commodities with the exception for the Entertainment (KS) and Reading (CvM). Moreover, the 1\% level tests reject separability in all cases, except for Reading (CvM) and Transportation (CvM). This suggests that Engel curves for these commodities may exhibit substantial heterogeneities in unobservables.
	
	\begin{table}
		\centering
		\caption{Testing separability of Engel curves. The table shows $m$ out of $n$ bootstrap p-values of Kolmogorov-Smirnov and Cram\'{e}r-von Mises tests for 13 commodities.}
		\begin{tabular}{lcc|clcc}
			Commodity & KS & CvM & Commodity & KS & CvM  \\	\hline
			Food home & 0.00 & 0.00 & Gas and oil & 0.00 & 0.00  \\
			Food away & 0.00 & 0.00 & Personal care & 0.00 & 0.00  \\
			Clothing & 0.00 & 0.00 & Health & 0.00 & 0.00  \\
			Tobacco & 0.00 & 0.00 & Insurance & 0.00 & 0.00 \\	
			Alcohol & 0.00 & 0.00 & Reading & 0.00 & 0.53  \\
			Trips & 0.00 & 0.00 & Transportation & 0.01 & 0.03  \\      
			Entertainment & 0.08 & 0.00 \\
		\end{tabular}
		\label{tab:engel}
	\end{table}  
	
	\section{Conclusions}\label{sec:conclusions}
	This paper offers a new perspective on the separability of unobservables in economic models with endogeneity. Starting from the nonseparable model where the instrumental variable is independent of unobservables, our first contribution is to develop a novel fully nonparametric separability test. The test is based on the estimation of a separable nonparametric IV regression and the verification of the independence restriction imposed by the nonseparable IV model. To obtain a large sample approximation to the distribution of our test statistics, we develop a novel uniform asymptotic expansions of the empirical distribution function of nonparametric IV residuals and obtain new results for the Tikhonov regularization in Sobolev spaces. We show that, despite the uncertainty coming from an ill-posed inverse nonparametric IV regression, the empirical distribution function of residuals and the residual-based independence empirical process still satisfy the Donsker central limit theorem. In contrast to the nonparametric regression without endogeneity, we find that the parameter uncertainty affects the asymptotic distribution of the residual-based independence tests, which are highly nonstandard. In our Monte Carlo experiments, we find that the bootstrap fails in approximating the distribution of the test statistics under the null hypothesis; hence we rely on the $m$ out of $n$ bootstrap (or subsampling) procedure to compute its critical values.
	
	Using the 2015 US Consumer Expenditure Survey data, we find that the $1\%$ level test rejects the separability of Engel curves for most of the commodities. This indicates that the Engel curves may be heterogeneous in unobservables and that the nonseparable modeling of Engel curves may be useful, see, e.g., \cite{blundell2017nonparametric} for the estimation of nonseparable demand functions. 
	
	The paper offers several other directions for future research. First, it might be interesting to test the separability of unobservables in other structural relations that are commonly estimated using the additively separable models in the empirical practice, such as the production function, the labor supply function, the demand function, or the wage equation. Second, given the plethora of residual-based specification tests for regression models without endogeneity, our results could also be used to develop similar tests for econometric models with endogeneity; see \cite{pardo2007testing} and \cite{escanciano2018asymptotic}.
	
	\section*{Acknowledgement}
	This work was supported by the French National Research Agency under Grant ANR-19-CE40-0013-01/ExtremReg project. We thank Ivan Canay, Tim Christensen, Elia Lapenta, Pascal Lavergne, Thierry Magnac, Nour Meddahi, and Ingrid Van Keilegom for helpful discussions. All remaining errors are ours.
	
	\newpage
	\setcounter{section}{0}
	\setcounter{equation}{0}
	\setcounter{table}{0}
	\setcounter{figure}{0}
	\renewcommand{\theequation}{A.\arabic{equation}}
	\renewcommand\thetable{A.\arabic{table}}
	\renewcommand\thefigure{A.\arabic{figure}}
	\renewcommand\thesection{A.\arabic{section}}
	\renewcommand\thetheorem{A.\arabic{theorem}}

	\begin{center}
		{\LARGE\textbf{APPENDIX: ADDITIONAL RESULTS AND PROOFS}}	
	\end{center}
	\bigskip
	\paragraph{Notation.} For two sequences $(a_n)_{n\in\Nn}$ and $(b_n)_{n\in\Nn}$, we denote $a_n\lesssim b_n$ if $a_n=O(b_n)$ and $a_n\sim b_n$ if both $a_n\lesssim b_n$ and $b_n\lesssim a_n$. For two sequences of random variables $(X_n)_{n\in\Nn}$ and $(Y_n)_{n\in\Nn}$, we denote $X_n\lesssim_P Y_n$ for $X_n = O_P(Y_n)$. For a bounded linear operator $T:\mathcal{X}\to\mathcal{Y}$ on normed spaces, we use $\|T\|_{\mathrm{op}} = \inf\{c\geq 0:\; \|Tx\|\leq c\|x\|,\forall x\in \mathcal{X}\}$ to denote its operator norm, where with some abuse of notation, we use $\|.\|$ to denote the norm of both spaces. 
	
	\section{Tikhonov regularization in Sobolev spaces}\label{appendix:regularization_sobolev}
	This section discusses convergence rates for the Tikhonov-regularized estimator in Sobolev spaces. The following result extends \cite{carrasco2013asymptotic}, Proposition 3.1 to the case of the unknown operator.
	\begin{theorem}\label{thm:risk_hs}
		Suppose that Assumption~\ref{as:hilbert_scale} is satisfied, $\|\hat T - T\|_{\mathrm{op}}^2 \lesssim_P\alpha_n$, and $2s\geq b-a$. Then for every $c\in[0,s]$
		\begin{equation*}
			\left\|\hat \varphi - \varphi\right\|^2_c \lesssim_P \alpha_n^{-\frac{a+c}{a+s}}\left\|\hat r - \hat T\varphi\right\|^2 + \alpha_n^{\frac{b-c}{a+s}}.
		\end{equation*}
	\end{theorem}
	It is worth emphasizing that this result is not specific to the nonparametric IV regression and can be applied to a generic ill-posed inverse problem $T\varphi = r$, where $(T,r)$ is estimated with $(\hat T,\hat r)$. Moreover, in the case of nonparametric IV regression, it can be easily applied to nonparametric/machine learning estimators $(\hat T,\hat r)$ other than the kernel smoothing. Next, we specialize the generic result of Theorem~\ref{thm:risk_hs} to the nonparametric IV regression with $(T,r)$ estimated via kernel smoothing, see equation~(\ref{eq:r_and_T}).
	\begin{corollary}\label{cor:rates_npiv}
		Suppose that Assumptions~\ref{as:hilbert_scale} and \ref{as:dgp} are satisfied, $\frac{1}{nh_n^{p+q}}\vee h_n^{2t} = O\left(\alpha_n\right)$, and $2s\geq b-a$. Then for every $c\in[0,s]$
		\begin{equation*}
			\|\hat{\varphi} - \varphi\|^2_c = O_P\left(\alpha_n^{-\frac{a+c}{a+s}}\left(\frac{1}{nh_n^q} + h_n^{2t}\right) + \alpha_n^{\frac{b-c}{a+s}}\right).
		\end{equation*}
	\end{corollary}
	
	\section{Distribution of nonparametric IV residuals}
	In this section, we present results on the weak convergence of the empirical distribution of nonparametric IV residuals. These results are used to obtain the large sample approximation to the distribution of independence tests and are of independent interest.
	\begin{theorem}\label{thm:residuals}
		Suppose that Assumptions~\ref{as:hilbert_scale}, \ref{as:dgp}, and \ref{as:smoothness} (i), and \ref{as:tuning} are satisfied. Then
		\begin{equation*}\label{eq:expansion}
			\sqrt{n}(\hat F_{\hat U}(u) - F_U(u)) = \frac{1}{\sqrt{n}}\sum_{i=1}^n\left\{\one_{\{U_i\leq u\}} - F_U(u) + U_i\left[T(T^*T)^{-1}f_{UZ}(u,.)\right](W_i)\right\} + o_P(1)
		\end{equation*}
		uniformly over $u\in\R$.
	\end{theorem}
	\begin{proof}
		By Lemma~\ref{lemma:remainder}, the following expansion holds uniformly in $u\in\R$
		\begin{equation*}
			\sqrt{n}(\hat F_{\hat U}(u) - F_U(u)) = \sqrt{n}(\hat F_U(u) - F_U(u)) + \sqrt{n}\left(\Pr\left(U\leq u + \hat\Delta(Z)|\mathscr{X}\right) - F_U(u)\right) + o_P(1).
		\end{equation*}
		By Taylor's theorem, there exists some $\tau\in[0,1]$ such that
		\begin{equation*}
			\begin{aligned}
				& \sqrt{n}\left(\Pr\left(U\leq u + \hat\Delta(Z)|\mathscr{X}\right) - \Pr(U\leq u)\right) \\
				& = \sqrt{n}\int\left\{\int_{-\infty}^{u+\hat\Delta(z)}f_{UZ}(v,z)\dx v - \int_{-\infty}^uf_{UZ}(v,z)\dx v\right\}\dx z \\
				& = \sqrt{n}\int\left\{f_{UZ}(u,z)\hat\Delta(z) + \frac{1}{2}\partial_uf_{UZ}(u+\tau\hat\Delta(z),z)\hat\Delta^2(z)\right\}\dx z \\
				& = \sqrt{n}\langle \hat\varphi - \varphi,f_{UZ}(u,.)\rangle + \sqrt{n}\frac{1}{2}\int\partial_uf_{UZ}(u+\tau\hat\Delta(z),z)\hat\Delta^2(z)\dx z \\
				& \triangleq T_{1n}(u) + T_{2n}(u).
			\end{aligned}
		\end{equation*}
		By Lemma~\ref{lemma:contribution} in the Supplementary Material,
		\begin{equation*}
			\begin{aligned}
				T_{1n}(u) = \frac{1}{\sqrt{n}}\sum_{i=1}^nU_i\left[T(T^*T)^{-1}f_{UZ}(u,.)\right](W_i) + o_P(1),
			\end{aligned}
		\end{equation*}
		while under Assumptions~\ref{as:smoothness} (i) and \ref{as:tuning}
		\begin{equation*}
			\begin{aligned}
				\|T_{2n}\|_\infty & \leq \|\partial_uf_{UZ}\|_\infty\sqrt{n}\|\hat\varphi -\varphi\|^2 = o_P(1).
			\end{aligned}
		\end{equation*}	
		
		Combining all estimates, we obtain uniformly in $u\in\R$
		\begin{equation*}
			\sqrt{n}(\hat F_{\hat U}(u) - F_U(u)) = \sqrt{n}(\hat F_U(u) - F_U(u)) + \frac{1}{\sqrt{n}}\sum_{i=1}^nU_i[T(T^*T)^{-1}f_{UZ}(u,.)](W_i) + o_P(1).
		\end{equation*}
	\end{proof}
	
	As a consequence of Theorem~\ref{thm:residuals}, we obtain the following Donsker-type central limit theorem for the empirical distribution of nonparametric IV residuals.
	\begin{corollary}\label{corr:clt}
		Suppose that assumptions of Theorem~\ref{thm:residuals} are satisfied. Then
		\begin{equation*}
			\sqrt{n}(\hat F_{\hat U} - F_U) \leadsto \mathbb{G}\qquad \mathrm{in}\qquad L_\infty(\R),
		\end{equation*}
		where $\mathbb{G}$ is a tight centered Gaussian process with uniformly continuous sample paths and the covariance function
		\begin{equation*}\small
			\begin{aligned}
				(u,u') & \mapsto F_U(u\wedge u') - F_U(u)F_U(u') + \E\left[U^2\left[T(T^*T)^{-1}f_{UZ}(u,.)\right](W)\left[T(T^*T)^{-1}f_{UZ}(u',.)\right](W)\right] \\
				+ & \E\left[\one_{\{U\leq u\}}U\left[T(T^*T)^{-1}f_{UZ}(u',.)\right](W) + \one_{\{U\leq u'\}}U\left[T(T^*T)^{-1}f_{UZ}(u,.)\right](W)\right].
			\end{aligned}
		\end{equation*}
	\end{corollary}
	\begin{proof}
		The process given in Theorem~\ref{thm:residuals} is an empirical process indexed by the following class of functions $\mathcal{F} = \left\{(v,w)\mapsto \one_{\{v\leq u\}} + v\left(T(T^*T)^{-1}f_{UZ}(u,.)\right)(w) ,\;u\in\R\right\}$, which is a sum of the Donsker class and $\mathcal{H} = \left\{(v,w)\mapsto v\left(T(T^*T)^{-1}f_{UZ}(u,.)\right)(w),\;u\in\R\right\}$. By \cite{van2000weak}, Example 2.10.5, it enough to show that $\mathcal{H}$ is Donsker. The former statement follows from the fact that under Assumption~\ref{as:hilbert_scale} in the Supplementary Material by \cite{engl1996regularization}, since for $\kappa-a > q/2$
		\begin{equation*}
			\sup_{u\in\R}\|T(T^*T)^{-1}f_{UZ}(u,.)\|_{\kappa-a} \lesssim \sup_{u\in\R}\|f_{UZ}(u,.)\|_{\kappa}\leq M<\infty,
		\end{equation*}
		where the last inequality follows under Assumption~\ref{as:smoothness} (i). Therefore, $\mathcal{H}\subset\{(v,w)\mapsto vg(w):\; g\in H^{\kappa-a}_M\}$, where $H_M^{\kappa-a}$ is a Sobolev ball of radius $M$. Since $\kappa>a+q/2$, this shows that the class $\mathcal{H}$ is Donsker; see \cite{nickl2007bracketing}, Corollaries 4 and 5. The covariance function simplifies since $\E[U|W]=0$.
	\end{proof}
	
	\section{Proofs of main results}\label{appendix:proofs}
	In this section we provide proofs of main results of the paper.
	\begin{proof}[Proof of Proposition~\ref{prop:separability}]
		Since $T:L_2(\R^p)\to L_2(\R^q)$ is injective, the nonparametric IV regression $\varphi\in L_2(\R^p)$ is unique. Therefore, $U=Y-\varphi(Z)$ is a well-defined unique random variable. If the model in equation (\ref{eq:nonseparable}) admits a separable representation, then since $\varepsilon\si W$
		\begin{equation*}
			\begin{aligned}
				\E[Y|W] & = \E[\psi(Z)+g(\varepsilon)|W] \\
				& = \E\left[\psi(Z)+\E g(\varepsilon)|W\right].
			\end{aligned}
		\end{equation*}
		Therefore, $\varphi(Z)=\psi(Z)+\E g(\varepsilon)$ by the injectivity of $T$, and whence $U=g(\varepsilon)-\E g(\varepsilon)$. This shows that $U\si W$ because $\varepsilon\si W$.
	\end{proof}
	
	\begin{proof}[Proof of Theorem~\ref{thm:main}]
		By Lemma~\ref{lemma:remainder2}, uniformly in $(u,w)$
		\begin{equation*}
			\mathbb{G}_n(u,w) = T_{1n}(u,w) + T_{2n}(u,w) - T_{3n}(u,w) + o_P(1),
		\end{equation*}
		where
		\begin{equation*}
			\begin{aligned}
				T_{1n}(u,w) & = \sqrt{n}\left(\hat F_{UW}(u,w) - \hat F_U(u)\hat F_W(w)\right), \\
				T_{2n}(u,w) & = \sqrt{n}\left(\Pr\left(U\leq u+\hat\Delta(Z), W\leq w|\mathscr{X}\right) - F_{UW}(u,w)\right), \\
				T_{3n}(u,w) & = \sqrt{n}\left(\Pr\left(U\leq u + \hat\Delta(Z)|\mathscr{X}\right) - F_U(u)\right)F_W(w).
			\end{aligned}
		\end{equation*}
		
		The first term is a classical independence empirical process
		{\small \begin{equation*}
				\begin{aligned}
					T_{1n}(u,w) & = \frac{1}{\sqrt{n}}\sum_{i=1}^n\left\{\one_{\{U_i\leq u,W_i\leq w\}} - \one_{\{U_i\leq u\}}F_W(w) - \one_{\{W_i\leq w\}}F_U(u) +F_U(u)F_W(w)\right\} \\
					& \qquad - \frac{1}{\sqrt{n}}\sum_{i=1}^n\left\{\one_{\{W_i\leq w\}} - F_W(w)\right\}\frac{1}{n}\sum_{i=1}^n\left\{\one_{\{U_i\leq u\}} - F_U(u)\right\} \\
					& = \frac{1}{\sqrt{n}}\sum_{i=1}^n\left\{\one_{\{U_i\leq u,W_i\leq w\}} - \one_{\{U_i\leq u\}}F_W(w) - \one_{\{W_i\leq w\}}F_U(u) +F_U(u)F_W(w)\right\} + o_P(1),
				\end{aligned}
		\end{equation*}}
		where the second line follows by the maximal inequality.
		
		Next, under Assumption~\ref{as:smoothness} (i), by Taylor's theorem, for some $\tau\in[0,1]$
		{\small \begin{equation*}
				\begin{aligned}
					T_{2n}(u,w) & = \sqrt{n} \iint^w\left\{\int_{-\infty}^{u+\hat\Delta(z)}f_{UZW}(\tilde u, z, \tilde w)\dx \tilde u - \int_{-\infty}^uf_{UZW}(\tilde u, z, \tilde w)\dx \tilde u\right\}\dx \tilde w\dx z \\
					& = \sqrt{n} \iint^w\left\{f_{UZW}(u,z,\tilde w)\hat\Delta(z) +  \frac{1}{2}\partial_u f_{UZW}(u+\tau\hat\Delta(z), z,\tilde w) \hat\Delta^2(z)\right\}\dx \tilde w\dx z \\
					& = \sqrt{n}\left\langle\hat\varphi - \varphi, \int^wf_{UZW}(u,.,\tilde w)\dx \tilde w\right\rangle + \frac{\sqrt{n}}{2}\iint^w\partial_uf_{UZW}(u+\tau\hat\Delta(z),z,\tilde w)\dx \tilde w\hat\Delta^2(z)\dx z \\
					& \triangleq S_{1n}(u,w) + S_{2n}(u,w).
				\end{aligned}
		\end{equation*}}
		Under Assumptions~\ref{as:smoothness} by Corollary~\ref{cor:rates_npiv}
		\begin{equation*}
			\begin{aligned}
				\|S_{2n}\|_\infty & \leq \sup_{w,u,z}\left|\int^w\partial_uf_{UZW}(u,z,\tilde w)\dx \tilde w\right|\sqrt{n}\left\|\hat{\varphi} - \varphi\right\|^2 = o_P(1)
			\end{aligned}
		\end{equation*}
		Similarly, we have uniformly in $(u,w)$
		\begin{equation*}
			\begin{aligned}
				T_{3n}(u,w) & = \sqrt{n}\langle\hat\varphi - \varphi,f_{UZ}(u,.)\rangle F_W(w) + o_P(1).
			\end{aligned}
		\end{equation*}
		Therefore, uniformly in $(u,w)\in\R\times\R^q$
		\begin{equation*}
			\begin{aligned}
				& T_{2n}(u,w)-T_{3n}(u,w) \\
				& = \sqrt{n}\int (\hat\varphi(z) - \varphi(z))\left\{\int^wf_{UZW}(u,z,\tilde w)\dx \tilde w - f_{UZ}(u,z)F_W(w)\right\}\dx z + o_P(1) \\
				& = \sqrt{n}\int(\hat{\varphi}(z) - \varphi(z))\rho(u,z,w)\dx z + o_P(1) \\
				& = \frac{1}{\sqrt{n}}\sum_{i=1}^nU_i\left(T(T^*T)^{-1}\rho(u,.,w)\right)(W_i) + o_P(1),
			\end{aligned}
		\end{equation*}
		where the last line follows by the same argument as in the proof of Theorem~\ref{thm:residuals} under Assumption~\ref{as:smoothness} (i).
	\end{proof}
	
	\begin{proof}[Proof of Proposition~\ref{corr:main}]
		$\mathbb{H}_n$ is an empirical process indexed by the class of functions
		\footnotesize{\begin{equation*}
				\mathcal{F} = \left\{(v,w)\mapsto \one_{\{v\leq \tilde v,w\leq \tilde w\}} - \one_{\{v\leq \tilde v\}}F_W(\tilde w) - \one_{\{w\leq \tilde w\}}F_U(\tilde v) + F_{UW}(\tilde v,\tilde w) + \delta_{\tilde v,\tilde w}(v,w):\;(\tilde v,\tilde w)\in\R^{1+q}\right\}.
		\end{equation*}} \normalsize
		By \cite{van2000weak}, Example 2.10.7 it suffices to show that each of the functions in the sum constitutes a Donsker class. To that end, recall first that the indicator functions are classical examples of Donsker classes. Therefore, all terms in $\mathcal{F}$, but the last one, are either Donsker or can be factored as Donsker classes and a deterministic bounded function not depending on the argument of the indicator function. Lastly, under Assumptions~\ref{as:hilbert_scale} (i) by \cite{engl1996regularization}, Corollary 8.22
		\begin{equation*}
			\|T(T^*T)^{-1}g(v,w,.)\|_{\kappa-a} \lesssim \sup_{(v,w)\in\R^{1+q}}\|g(v,w,.)\|_\kappa \leq M<\infty,
		\end{equation*}
		where the latter follows under Assumption~\ref{as:smoothness} (ii). Therefore, we obtain that $\{(v,w)\mapsto v(T(T^*T)^{-1}g(\tilde v,\tilde w,.))(w):\;\tilde v\in\R, \tilde w\in\R^q \}\subset \{(v,w)\mapsto vg(w):\; g\in H^{\kappa-a}_M\}$, where $H_M^{\kappa-a}$ is a Sobolev ball of radius $M$. Since $\kappa>a+q/2$, this shows that $\mathcal{F}$ is Donsker; see \cite{nickl2007bracketing}, Corollaries 4 and 5.
	\end{proof}
	
	\begin{proof}[Proof of Corollary~\ref{cor:asympt_distributions}]
		Since under $H_0$, $\mathbb{G}_n\leadsto \mathbb{H}$ by Proposition~\ref{corr:main}, the asymptotic distribution of $T_{\infty,n}$ under $H_0$ is readily obtained by the continuous mapping theorem; see \cite{van2000weak}, Theorem 1.3.6. For the Cram\'{e}r-von Mises statistics, write
		\begin{equation*}
			T_{2,n} = \iint \mathbb{H}^2(u,w)\dx F_{UW}(u,w) + R_{1n} + R_{2n}
		\end{equation*}
		with
		\begin{equation*}
			\begin{aligned}
				R_{1n} & = \iint\left\{\mathbb{G}_n^2(u,w) - \mathbb{H}^2(u,w)\right\}\dx \hat F_{\hat UW}(u,w) \\
				R_{2n} & = \iint\mathbb{H}^2(u,w)\dx[\hat F_{\hat UW}(u,w) - F_{UW}(u,w)].
			\end{aligned}
		\end{equation*}
		By Proposition~\ref{corr:main}, under $H_0$, $\mathbb{G}_n\leadsto \mathbb{H}$ and $\sqrt{n}(\hat F_{\hat UW}(u,w) - F_{UW}(u,w))$ also converges weakly by Proposition~\ref{corr:main} and Theorem~\ref{corr:clt}, whence by the Skorokhod construction
		\begin{equation}\label{eq:skorokhod}
			\begin{aligned}
				n^{-1/2}\sup_{u,w}\left|\mathbb{G}_n(u,w)\right|\xrightarrow{\mathrm{a.s.}}0\qquad \text{and}\qquad \sup_{u,w}\left|\hat F_{\hat UW}(u,w) - F_{UW}(u,w)\right|\xrightarrow{\mathrm{a.s.}}0.
			\end{aligned}
		\end{equation}
		The first expression in Eq.~\ref{eq:skorokhod} implies that $R_{1n}\xrightarrow{\mathrm{a.s.}}0$. Since $\mathbb{H}$ has a.s. bounded and continuous trajectories, the second expression in Eq.~\ref{eq:skorokhod} in conjunction with the Helly-Bray theorem show that $R_{2n}\xrightarrow{\mathrm{a.s.}}0$. Therefore, the asymptotic distribution of the Cram\'{e}r-von Mises test follows by the continuous mapping theorem.
		
		Under the fixed alternative hypothesis, since $\E[U|W]=0$, by Theorem~\ref{thm:main}, the Glivenko-Cantelli theorem, and a similar argument we obtain
		\begin{equation*}
			\begin{aligned}
				n^{-1/2}T_{2,n} & = \iint|n^{-1/2}\mathbb{G}_n(u,w)|^2\dx \hat F_{\hat UW}(u,w) \xrightarrow{\mathrm{a.s.}} 2d_2>0 \\
				n^{-1/2}T_{\infty,n} & = \sup_{u,w}|n^{-1/2}\mathbb{G}_n(u,w)| \xrightarrow{\mathrm{a.s.}} 2d_\infty>0.
			\end{aligned}
		\end{equation*}
		Therefore, by Slutsky's theorem $T_{2,n} \xrightarrow{\mathrm{a.s.}} \infty$ and $T_{\infty,n} \xrightarrow{\mathrm{a.s.}} \infty$, which proves the second statement. For the local alternatives, note that
		\begin{equation*}
			\begin{aligned}
				\E[h_{u,w}(U,W)] & = 2\left(F_{UW}(u,w) - F_U(u)F_W(w)\right) = 2n^{-1/2}H(u,w).
			\end{aligned}
		\end{equation*}
		Therefore, by Corollary~\ref{corr:main} and continuous mapping theorem
		\begin{equation*}
			\begin{aligned}
				T_{\infty,n} & = \sup_{u,w}|\mathbb{G}_n(u,w)| \\
				& = \sup_{u,w}|\mathbb{G}_n(u,w) - \sqrt{n}\E[h_{u,w}(U,W)] + 2H(u,w)| \\
				& \leadsto \sup_{u,w}|\mathbb{H}(u,w)  + 2H(u,w)|.
			\end{aligned}
		\end{equation*}
		For the Cram\'{e}r-von Mises statistics, write
		\begin{equation*}
			\begin{aligned}
				T_{2,n} & = \iint|\mathbb{H}(u,w) + 2H(u,w)|^2\dx F_{UW}(u,w) + S_{1n} + S_{2n},
			\end{aligned}
		\end{equation*}
		where
		\begin{equation*}
			\begin{aligned}
				S_{1n} & = \iint \left\{\left|\mathbb{G}_n(u,w) - \sqrt{n}\E[h_{u,w}(U,W)] + 2H(u,w)\right|^2 - |\mathbb{H}(u,w) + 2H(u,w)|^2 \right\}\dx \hat F_{\hat UW}(u,w) \\
				S_{2n} & = \iint|\mathbb{H}(u,w) + 2H(u,w)|^2\dx\left[\hat F_{\hat UW}(u,w) - F_{UW}(u,w)\right].
			\end{aligned}
		\end{equation*}
		Therefore, the result follows by Proposition~\ref{corr:main} and the same argument as under $H_0$ with the only difference that now we have the bias $2H$ in the limiting distribution.
	\end{proof}
	
	\section{Auxiliary technical results}
	In this section, we provide several auxiliary technical results.
	\begin{lemma}\label{lemma:remainder}
		Suppose that Assumption~\ref{as:hilbert_scale}, \ref{as:dgp}, \ref{as:smoothness}, and \ref{as:tuning}. Then
		\begin{equation}\label{eq:remainder}
			\sup_u\left|\hat F_{\hat U}(u)  - \hat F_{U}(u) - \Pr(U\leq u + \hat\Delta(Z)|\mathscr{X}) + F_{U}(u)\right| = o_P\left(n^{-1\slash 2}\right),
		\end{equation}
		where $\hat\Delta = \hat{\varphi} - \varphi$ and $\mathscr{X}=(Y_i,Z_i,W_i)_{i=1}^\infty$.
	\end{lemma}
	\begin{proof}
		The main idea of the proof is to embed the process inside the supremum into an empirical process indexed by $u$ and a Sobolev ball containing $\hat \Delta$ with a probability tending to one. We first show that the process is Donsker, whence the supremum in Eq.~\ref{eq:remainder} is $O_P(n^{-1/2})$. Finally, the required $o_P(n^{-1/2})$ order will follow from the fact that the process is degenerate.
		
		Let $H_M^c$ be a ball of radius $M<\infty$ in the Sobolev space $H^c(\R^p)$. For $u\in\R$ and $\Delta\in H^c_M$, define $f_{u,\Delta}(U,Z)=\one_{(-\infty,u+\Delta(Z)]}(U)$, $\mathcal{G}_1 = \left\{f_{u,\Delta}:\; u\in\R,\Delta\in H^c_M(\R^p) \right\}$, $\mathcal{G}_2 =  \left\{f_{u,0}:\; u\in\R\right\}$, and $\mathcal{G} = \mathcal{G}_1 - \mathcal{G}_2$.
		Note that $\mathcal{G}_2$ is a classical Donsker class of indicator functions. If we can show that $\mathcal{G}_1$ is Donsker, then $\mathcal{G}$ will be Donsker as a sum of two Donsker classes; see \cite{van2000weak}, Theorem 2.10.6. To this end, we check that the bracketing entropy condition is satisfied for $\mathcal{G}_1$. 
		
		By  \cite{nickl2007bracketing}, Corollary 4 the bracketing number of $H_M^c$ satisfies $\log N_{[\;]}(\varepsilon, H_M^c,\|.\|_{L^2_Z}) \lesssim \varepsilon^{-p/c}$, where $(L^2_Z,\|.\|_{L^2_Z})$ denotes the space of functions, square-integrable with respect to $f_Z$. Put $M_\varepsilon = N_{[\;]}(\varepsilon, H_M^c,\|.\|_{L^2_Z})$ and fix $u\in\R$. Let $\left[\underline\Delta_j,\overline\Delta_j\right]_{j=1}^{M_\varepsilon}$ be a collection of  $\varepsilon$-brackets for $H^c_M$, i.e., for any $\Delta\in H^c_M$, there exists $1\leq j\leq M_\varepsilon$ such that $\underline\Delta_j\leq \Delta\leq \overline{\Delta}_j$ and $\left\|\overline\Delta_j-\underline\Delta_j\right\|_{L_Z^2}\leq\varepsilon$, and whence $\one_{\left(-\infty, u + \underline\Delta_j\right]} \leq \one_{\left(-\infty, u + \Delta\right]} \leq \one_{\left(-\infty, u + \overline\Delta_j\right]}$. Now for each $1\leq j\leq M_\varepsilon$, partition the real line into intervals defined by grids of points $-\infty=\underline u_{j,1}<\underline u_{j,2}<\dots<\underline u_{j,M_{1\varepsilon}}=\infty$ and $-\infty=\overline u_{j,1}<\overline u_{j,2}<\dots<\overline u_{j,M_{2\varepsilon}}=\infty$, so that each segment has probabilities
		\begin{equation*}
			\begin{aligned}
				\Pr\left(U-\underline\Delta_j(Z)\leq \underline u_{j,k}\right) - \Pr\left(U-\underline \Delta_{j}(Z)\leq \underline u_{j,k-1}\right) & \leq \varepsilon^2/2,\qquad 2\leq k\leq \frac{2}{\varepsilon^2}\triangleq M_{1\varepsilon}, \\
				\Pr\left(U-\overline\Delta_j(Z)\leq \overline u_{j,k}\right) - \Pr\left(U-\overline \Delta_{j}(Z)\leq \overline u_{j,k-1}\right) & \leq \varepsilon^2/2, \qquad 2\leq k\leq \frac{2}{\varepsilon^2} \triangleq M_{2\varepsilon}.\\
			\end{aligned}
		\end{equation*}
		Denote the largest $\underline u_{j,k}$ such that $\underline u_{j,k}\leq u$ by $\underline{u_{j}^*}$ and the smallest $\overline u_{j,k}$ such that $u\leq \overline u_{jk}$ by $\overline u_{j}^*$. Consider the following family of brackets $\left[\one_{\left(-\infty, \underline u_{j}^*+\underline\Delta_j\right]},\one_{\left(-\infty, \overline u_{j}^*+\overline\Delta_j\right]}\right]_{j=1}^{M_\varepsilon}.$ Under Assumption~\ref{as:dgp} (ii)
		\begin{equation*}
			\begin{aligned}
				\left\|\one_{\left(-\infty, \overline u_{j}^*+\overline\Delta_j\right]} - \one_{\left(-\infty, \underline u_{j}^*+\underline\Delta_j\right]}\right\|_{L_Z^2}^2 & = \Pr\left(\underline u_{j}^*+\underline\Delta_j(Z)\leq U\leq \overline u_j^* + \overline\Delta_j(Z)\right) \\
				& \leq \Pr\left(u+\underline\Delta_j(Z) \leq U\leq u + \overline\Delta_j(Z)\right)  + \varepsilon^2 \\
				& = \int\left\{ \int_{u+\underline\Delta_j(z)}^{u+\overline\Delta_j(z)}f_{U|Z}(u|z)\dx u\right\}f_Z(z)\dx z + \varepsilon^2 \\
				& \leq \left\|\overline\Delta_j - \underline\Delta_j\right\|_{L^2_Z}\|f_{U|Z}\|_\infty + \varepsilon^2= O\left(\varepsilon^2\right).
			\end{aligned}
		\end{equation*}
		Therefore, we constructed brackets of size $O(\varepsilon)$, covering $\mathcal{G}_1$, and we have used at most $O\left(\varepsilon^{-2}M_\varepsilon\right)$ such brackets. Since $c>p/2$, we have $\int_0^1 \sqrt{\log N_{[\;]}(\varepsilon,\mathcal{G},\|.\|_{L_Z^2})}\dx \varepsilon < \infty$. This shows that the empirical process $\sqrt{n}(P_n - P)g,g\in\mathcal{G}$ is Donsker, hence, asymptotically equicontinuous; see \cite{van2000weak}, Theorem 1.5.7. Then for any $\varepsilon>0$
		\begin{equation}\label{eq:equicontinuity}
			\lim_{\delta\downarrow 0}\limsup_{n\to\infty}\mathrm{Pr}^*\left(\sup_{f,g\in\mathcal{G}:\;\rho(f-g)<\delta}|\sqrt{n}(P_n - P)(f-g)|>\varepsilon\right) = 0,
		\end{equation}
		where $\mathrm{Pr}^*$ denotes the outer probability measure. 
		
		Next, we show that for every $u\in\R$, $\rho^2(\hat f_u) = \E[\hat f_u^2] - (\E[\hat f_u])^2 = o_P(1)$	with  $\hat f_u = \one_{(-\infty,u+\hat\Delta(Z)]}(U) - \one_{(-\infty,u]}(U)$, where the expectation is computed with respect to $(U,Z)$ only. Indeed,
		\begin{equation*}
			\begin{aligned}
				\E[\hat f_u] & = \Pr(u\leq U\leq u+\hat \Delta(Z)|\mathscr{X}) \\
				& = \int \int_u^{u+\hat \Delta(z)}f_{U|Z}(v|z)\dx vf_Z(z)\dx z\\
				& \leq\|f_{U|Z}\|_\infty\|f_Z\| \|\hat \Delta\|  = o_P(1),
			\end{aligned}
		\end{equation*}	
		where the third line follows by the Cauchy-Schwartz inequality and Corollary~\ref{cor:rates_npiv} under Assumptions~\ref{as:hilbert_scale}, \ref{as:dgp}, and \ref{as:tuning}. Similarly,
		\begin{equation*}
			\begin{aligned}
				\E[\hat f_u^2] & = \Pr(U\leq u+\hat \Delta(Z)|\mathscr{X}) + \Pr(U\leq u) - 2\Pr(U\leq (u+\hat \Delta(Z))\wedge u|\mathscr{X}) \\
				& \leq \iint_{u}^{u+\hat \Delta(z)}f_{U|Z}(v|z)\dx vf_Z(z)\dx z \lesssim \|\hat\Delta\| = o_P(1).\\
			\end{aligned}
		\end{equation*}
		
		Lastly, let $\|\hat \nu_n\|_\infty$ denote the supremum in Eq~\ref{eq:remainder}. Then
		\begin{equation*}
			\begin{aligned}
				\mathrm{Pr}^*(\sqrt{n}\|\hat \nu_n\|_\infty > \varepsilon) & \leq \mathrm{Pr}^*\left(\sqrt{n}\|\hat \nu_n\|_\infty>\varepsilon,\rho(\hat f_u)< \delta,\hat \Delta\in H_M^c\right) + \mathrm{Pr}^*\left(\rho(\hat f_u)\geq \delta\right) \\
				& \qquad + \mathrm{Pr}^*\left(\hat \Delta\not\in H^c_M\right),
			\end{aligned}
		\end{equation*}
		where the second probability tends to zero as we have just shown and the last probability tends to zero since under the maintained assumptions, by Corollary~\ref{cor:rates_npiv}, $\|\hat\varphi - \varphi\|_c = o_P(1)$. Therefore, it follows from the asymptotic equicontinuity in Eq.~\ref{eq:equicontinuity} that $\limsup_{n\to\infty}\mathrm{Pr}^*(\sqrt{n}\|\hat \nu_n\|_\infty > \varepsilon)=0$, which concludes the proof.
	\end{proof}

	\begin{lemma}\label{lemma:remainder2}
		Suppose that Assumptions~\ref{as:hilbert_scale}, \ref{as:dgp}, \ref{as:smoothness}, and \ref{as:tuning} are satisfied. Then uniformly over $(u,w)\in\R\times\R^q$
		\begin{equation*}
			(\hat F_{\hat U}(u) - \hat F_U(u))\hat F_W(w) - \left(\Pr(U \leq u + \hat\Delta(Z)|\mathscr{X}) + F_U(u)\right)F_W(w) = o_P\left(n^{-1\slash 2}\right)
		\end{equation*}
		and	
		\begin{equation*}
			\hat F_{\hat UW}(u,w)  - \hat F_{UW}(u,w) - \Pr(U\leq u+\hat\Delta(Z),W\leq w|\mathscr{X}) + F_{UW}(u,w) = o_P\left(n^{-1\slash 2}\right).
		\end{equation*}	
		where $\hat\Delta = \hat\varphi - \varphi$ and $\mathscr{X}=(Y_i,Z_i,W_i)_{i=1}^\infty$,
	\end{lemma}
	\begin{proof}
		Note that the first expression and the expression in the statement of Lemma~\ref{lemma:remainder} multiplied by $F_W$ differ only by
		\begin{equation*}
			(\hat F_{\hat U}(u) - F(u))(\hat F_W(w) - F_W(w)),
		\end{equation*}
		which is $O_P(n^{-1})$ by Corollary~\ref{corr:clt} and the classical Donsker central limit theorem. By Lemma~\ref{lemma:remainder}, we obtain the first statement since $F_W$ is uniformly bounded by one.
		
		The proof of the second statement is similar to the proof of Lemma~\ref{lemma:remainder} and is omitted.
	\end{proof}

\newpage
\spacingset{1.45} 
\setcounter{page}{1}
\setcounter{section}{0}
\setcounter{equation}{0}
\setcounter{table}{0}
\setcounter{figure}{0}
\renewcommand{\theequation}{B.\arabic{equation}}
\renewcommand\thetable{B.\arabic{table}}
\renewcommand\thefigure{B.\arabic{figure}}
\renewcommand\thesection{B.\arabic{section}}
\renewcommand\thepage{Online Appendix - \arabic{page}}
\renewcommand\thetheorem{B.\arabic{theorem}}

\section{Additional proofs and auxiliary results}
This section contains proofs of several results from the main part of the paper as well as several auxiliary result.
\begin{proof}[Proof of Theorem A.1]
	Decompose
	\begin{equation*}
		\hat{\varphi} - \varphi = I_n + II_n + III_n + IV_n + V_n,
	\end{equation*}
	with
	\begin{equation*}
		\begin{aligned}
			I_n & = L^{-s}(\alpha_n I + T_s^*T_s)^{-1}T_s^*(\hat r - \hat T\varphi), \\
			II_n & = L^{-s}(\alpha_n I + T_s^*T_s)^{-1}(\hat T_s^* - T_s^*)(\hat r - \hat T\varphi), \\
			III_n & = L^{-s}\left[(\alpha_n I + \hat T_s^*\hat T_s)^{-1} - (\alpha_n I + T_s^*T_s)^{-1}\right]\hat T_s^*(\hat r - \hat T\varphi), \\ 
			IV_n & = L^{-s}(\alpha_nI + \hat T^*_s\hat T_s)^{-1}\hat T^*_s\hat T_sL^s\varphi - L^{-s}(\alpha_n I + T^*_sT_s)^{-1}T^*_sT_sL^s\varphi, \\
			V_n & = L^{-s}(\alpha_n I + T^*_sT_s)^{-1}T_s^*T_sL^s\varphi - \varphi.
		\end{aligned}
	\end{equation*}
	For the first term
	\begin{equation*}
		\begin{aligned}
			\|I_n\|_c^2 & = \left\|(\alpha_n I + T^*_sT_s)^{-1}T^*_s(\hat r - \hat T\varphi)\right\|^2_{c-s} \\
			& \lesssim \left\|(T^*_sT_s)^{\frac{s-c}{2(a+s)}}(\alpha_n I + T^*_sT_s)^{-1}T^*_s(\hat r - \hat T\varphi)\right\|^2 \\
			& \leq \left\|(T^*_sT_s)^{\frac{s-c}{2(a+s)}}(\alpha_n I + T^*_sT_s)^{-1}T^*_s\right\|^2_{\mathrm{op}}\left\|\hat r - \hat T\varphi\right\|^2 \\
			& \leq \sup_\lambda\left|\frac{\lambda^\frac{2s+a-c}{2(a+s)}}{\alpha_n + \lambda}\right|^2 \left\|(\hat r - \hat T\varphi)\right\|^2 \\
			& \lesssim \alpha_n^{-\frac{a+c}{a+s}}\left\|\hat r - \hat T\varphi\right\|^2,
		\end{aligned}
	\end{equation*}
	where the second line follows by \cite{engl1996regularization}, Corollary 8.22 with $\nu=(s-c)/(a+s)\leq 1$; the third line by the definition of operator norm; the fourth line by the isometry of functional calculus; and the last since $\sup_\lambda|\lambda^d/(\alpha_n+\lambda)| \lesssim \alpha_n^{d-1}$ for all $d\in[0,1]$.
	
	Similarly, since for bounded linear operators $A$ and $B$, $\|AB\|_{\mathrm{op}}\leq \|A\|_{\mathrm{op}}\|B\|_{\mathrm{op}}$,
	\begin{equation*}
		\begin{aligned}
			\|II_n\|_c^2 & = \left\|(\alpha_n I + T_s^*T_s)^{-1}(\hat T_s^* - T_s^*)(\hat r - \hat T\varphi)\right\|^2_{c-s} \\
			& \lesssim \left\|(T^*_sT_s)^{\frac{s-c}{2(a+s)}}(\alpha_n I + T^*_sT_s)^{-1}\right\|_{\mathrm{op}}^2\|\hat T^* - T^*\|^2_{\mathrm{op}}\left\|\hat r - \hat T\varphi\right\|^2 \\
			& \lesssim_P \alpha_n^{-\frac{2a+s+c}{a+s}}\alpha_n\left\|\hat r - \hat T\varphi\right\|^2 \\
			& \lesssim \alpha_n^{-\frac{a+c}{a+s}}\left\|\hat r - \hat T\varphi\right\|^2.
		\end{aligned}
	\end{equation*}
	
	Next, since $L^s\varphi\in H^{b-s}$ and $s\geq (b-a)/2$, by \cite{engl1996regularization}, Corollary 8.22, there exists $\psi\in L_2$ such that $L^s\varphi = (T_s^*T_s)^{\frac{b-s}{2(a+s)}}\psi$. Therefore,
	\begin{equation*}
		\begin{aligned}
			\|V_n\|_c^2 & = \left\|(\alpha_nI+ T^*_sT_s)^{-1}T^*_sT_sL^s\varphi - L^s\varphi\right\|_{c-s} \\
			& = \left\|\alpha_n(\alpha_n I + T^*_sT_s)^{-1}L^s\varphi\right\|^2_{c-s} \\
			& \lesssim \left\|\alpha_n(T_s^*T_s)^\frac{s-c}{2(a+s)}(\alpha_n I + T^*_sT_s)^{-1}(T_s^*T_s)^\frac{b-s}{2(a+s)}\psi\right\|^2\\
			& \lesssim \left\|\alpha_n(T_s^*T_s)^\frac{s-c}{2(a+s)}(\alpha_n I + T^*_sT_s)^{-1}(T_s^*T_s)^\frac{b-s}{2(a+s)}\right\|^2_{\mathrm{op}} \\
			& \leq \sup_{\lambda}\left|\frac{\alpha_n\lambda^{\frac{b-c}{2(a+s)}}}{\alpha_n+\lambda}\right|^2 \lesssim \alpha_n^{\frac{b-c}{a+s}}.
		\end{aligned}
	\end{equation*}

	Next, decompose
	\begin{equation*}
		\begin{aligned}
			\|III_n\|_c^2 & = \left\|\left[(\alpha_n I + T_s^*T_s)^{-1} - (\alpha_n I + \hat T_s^*\hat T_s)^{-1}\right]\hat T_s^*(\hat r - \hat T\varphi)\right\|^2_{c-s} \\
			& = \left\|(\alpha_n I + T_s^*T_s)^{-1}(\hat T_s^*\hat T_s - T_s^*T_s)(\alpha_n I + \hat T_s^*\hat T_s)^{-1}\hat T_s^*(\hat r - \hat T\varphi)\right\|^2_{c-s} \\
			& \leq 2R_{1n} + 2R_{2n}
		\end{aligned}
	\end{equation*}
	with
	\begin{equation*}
		\begin{aligned}
			R_{1n} & = \left\|(\alpha_n I + T_s^*T_s)^{-1}T_s^*(\hat T_s - T_s)(\alpha_n I + \hat T_s^*\hat T_s)^{-1}\hat T_s^*(\hat r - \hat T\varphi) \right\|^2_{s-c} \\
			& \lesssim \left\|(T^*_sT_s)^{\frac{s-c}{2(a+s)}}(\alpha_n I + T_s^*T_s)^{-1}T_s^*\right\|^2_{\mathrm{op}}\|\hat T_s - T_s\|^2_{\mathrm{op}}\|(\alpha_n I + \hat T_s^*\hat T_s)^{-1}\hat T_s^*\|_{\mathrm{op}}^2\left\|\hat r - \hat T\varphi\right\|^2 \\
			& \leq \left\|(T^*_sT_s)^{\frac{s-c}{2(a+s)}}(\alpha_n I + T_s^*T_s)^{-1}T_s^*\right\|^2_{\mathrm{op}}\alpha_n\frac{1}{\alpha_n} \left\|\hat r - \hat T\varphi\right\|^2\\
			& \lesssim_P \alpha_n^{-\frac{a+c}{a+s}}\left\|\hat r - \hat T\varphi \right\|^2 \\
		\end{aligned}
	\end{equation*}
	and
	\begin{equation*}
		\begin{aligned}
			R_{2n} & = \left\|(T^*_sT_s)^{\frac{s-c}{2(a+s)}}(\alpha_n I + T_s^*T_s)^{-1}(\hat T_s^* - T_s^*)\hat T_s(\alpha_n I + \hat T_s^*\hat T_s)^{-1}\hat T_s^*(\hat r - \hat T\varphi) \right\|^2 \\
			& \leq \left\|(T^*_sT_s)^{\frac{s-c}{2(a+s)}}(\alpha_n I + T_s^*T_s)^{-1}\right\|^2_{\mathrm{op}}\|\hat T_s^* - T_s^*\|^2_{\mathrm{op}}\left\|\hat T_s(\alpha_n I + \hat T_s^*\hat T_s)^{-1}\hat T_s^*\right\|^2_{\mathrm{op}}\left\|\hat r - \hat T\varphi\right\|^2 \\
			& \leq \left\|(T^*_sT_s)^{\frac{s-c}{2(a+s)}}(\alpha_n I + T_s^*T_s)^{-1}\right\|^2_{\mathrm{op}}\alpha_n \left\|\hat r - \hat T\varphi\right\|^2 \\
			& \lesssim_P \alpha_n^{-\frac{2a+c+s}{a+s}}\alpha_n\left\|\hat r - \hat T\varphi\right\|^2 \\
			& \lesssim_P \alpha_n^{-\frac{a+c}{a+s}}\left\|\hat r - \hat T\varphi \right\|^2.
		\end{aligned}
	\end{equation*}
	
	Similarly, decompose
	\begin{equation*}
		\begin{aligned}		
			\|IV_n\|_c^2 & = \left\|\alpha_n\left[(\alpha_n I + \hat T^*_s\hat T_s)^{-1} - (\alpha_n I + \hat T^*_s\hat T_s)^{-1}\right]L^s\varphi\right\|^2_{c-s} \\
			& \lesssim \left\|(\alpha_nI + \hat T^*_s\hat T_s)^{-1}\left(\hat T^*_s\hat T_s - T_s^*T_s\right)\alpha_n(\alpha_nI + T^*_sT_s)^{-1}L^s\varphi\right\|^2_{c-s}\\
			& \leq 2 S_{1n} + 2 S_{2n} 
		\end{aligned}
	\end{equation*}
	with $S_{1n}$ and $S_{2n}$ defined below. In particular, 
	\begin{equation*}
		\begin{aligned}
			S_{1n} & = \left\|(\alpha_nI + \hat T^*_s\hat T_s)^{-1}\hat T^*_s\left(\hat T_s - T_s\right)\alpha_n(\alpha_nI + T^*_sT_s)^{-1}L^s\varphi\right\|^2_{c-s} \\
			& \lesssim \left\|(\alpha_nI + \hat T^*_s\hat T_s)^{-1}\hat T^*_s\left(\hat T_s - T_s\right)\alpha_n(\alpha_nI + T^*_sT_s)^{-1}L^s\varphi\right\|^2 \\
			& \lesssim \left\|(\alpha_nI + \hat T^*_s\hat T_s)^{-1}\hat T^*_s\right\|^2_{\mathrm{op}}\|\hat T - T\|^2_{\mathrm{op}}\left\|L^{-s}\alpha_n(\alpha_nI + T^*_sT_s)^{-1}L^s\varphi\right\|^2 \\
			& \leq \left\|\alpha_n(\alpha_nI + T^*_sT_s)^{-1}L^s\varphi\right\|^2_{-s} \\
			& \lesssim \left\|\alpha_n(T_s^*T_s)^\frac{s}{2(a+s)}(\alpha_n I + T^*_sT_s)^{-1}(T_s^*T_s)^\frac{b-s}{2(a+s)}\psi\right\|^2\\ 
			& \lesssim \sup_\lambda\left|\frac{\alpha_n\lambda^{\frac{b}{2(a+s)}}}{\alpha_n + \lambda}\right|^2 \lesssim \alpha_n^\frac{b}{a+s},
		\end{aligned}
	\end{equation*}
	where the last two lines follow by \cite{engl1996regularization}, Corollary 8.22 with $\nu=s/(a+s)\leq 1$ and previous computations. Similarly,
	\begin{equation*}
		\begin{aligned}
			S_{2n} & = \left\|(\alpha_nI + \hat T^*_s\hat T_s)^{-1}\left(\hat T^*_s - T^*_s\right)\alpha_nT_s(\alpha_nI + T^*_sT_s)^{-1}L^s\varphi\right\|^2_{c-s} \\	
			& \leq \left\|(T^*_sT_s)^{\frac{s-c}{2(a+s)}}(\alpha_nI + \hat T^*_s\hat T_s)^{-1}\right\|^2_{\rm op}\|\hat T^*_s - T^*_s\|^2_{\mathrm{op}}\left\|\alpha_nT_s(\alpha_nI + T^*_sT_s)^{-1}(T_s^*T_s)^\frac{b-s}{2(a+s)}\psi\right\|^2 \\
			& \lesssim_P \alpha_n^{-\frac{2a+s+c}{a+s}}\|\hat T - T\|^2_{\mathrm{op}}\alpha_n^{\frac{b+a}{a+s}} \\
			& \lesssim_P \alpha_n^{\frac{b-c}{a+s}}.
		\end{aligned}
	\end{equation*}
	The result follows from combining all estimates.
\end{proof}

\begin{proof}[Proof of Corollary A.1.1]
	By the Cauchy-Schwartz inequality
	\begin{equation*}
		\begin{aligned}
			\|\hat T - T\|^2_{\mathrm{op}} & \leq \left\|\hat f_{ZW} - f_{ZW}\right\|^2 \\
			& = O_P\left(\frac{1}{nh_n^{p+q}} + h_n^{2t}\right),
		\end{aligned}
	\end{equation*}
	where the second line follows from the well-known risk bound; see, e.g., \cite{gine2015mathematical}, p. 403-404 under Assumption~\ref{as:dgp}.
	Therefore, by Theorem~\ref{thm:risk_hs}
	\begin{equation*}
		\left\|\hat \varphi - \varphi\right\|^2_c \lesssim_P \alpha_n^{-\frac{a+c}{a+s}}\left\|\hat r - \hat T\varphi\right\|^2 + \alpha_n^{\frac{b-c}{a+s}}.
	\end{equation*}	
	The proof of
	\begin{equation*}
		\left\|\hat r - \hat T\varphi\right\|^2 = O_P\left(\frac{1}{nh_n^q} + h_n^{2t}\right)
	\end{equation*}
	under Assumption~\ref{as:dgp} can be found in \cite{babii2017completeness}.
\end{proof}

\begin{lemma}\label{lemma:contribution}
	Suppose that Assumptions~\ref{as:hilbert_scale}, \ref{as:dgp}, \ref{as:smoothness}, and \ref{as:tuning}, and  are satisfied. Then
	\begin{equation*}
		\langle\hat\varphi - \varphi,f_{UZ}(u,.)\rangle = \frac{1}{\sqrt{n}}\sum_{i=1}^nU_i\left[T(T^*T)^{-1}f_{UZ}(u,.)\right](W_i) + o_P(1)
	\end{equation*}
\end{lemma}
\begin{proof}
	Similarly to the proof of Theorem~\ref{thm:risk_hs}, decompose
	\begin{equation*}
		\sqrt{n}\langle \hat\varphi - \varphi,f_{UZ}(u,.)\rangle \triangleq I_n(u) + II_n(u) + III_n(u)
	\end{equation*}
	with
	\begin{equation*}
		\begin{aligned}
			I_n(u) & = \sqrt{n}\left\langle L^{-s}(\alpha_n I + T_s^*T_s)^{-1}T_s^*(\hat r - \hat T\varphi), f_{UZ}(u,.)\right\rangle, \\
			II_n(u) & = \sqrt{n}\left\langle L^{-s}(\alpha_n I + T_s^*T_s)^{-1}(\hat T_s^* - T_s^*)(\hat r - \hat T\varphi), f_{UZ}(u,.)\right\rangle, \\
			III_n(u) & = \sqrt{n}\left\langle L^{-s}\left[(\alpha_n I + \hat T_s^*\hat T_s)^{-1} - (\alpha_n I + T_s^*T_s)^{-1}\right]\hat T_s^*(\hat r - \hat T\varphi), f_{UZ}(u,.)\right\rangle, \\
			IV_n(u) & = \sqrt{n}\left\langle L^{-s}(\alpha_nI + \hat T^*_s\hat T_s)^{-1}\hat T^*_s\hat T_sL^s\varphi - L^{-s}(\alpha_n I + T^*_sT_s)^{-1}T^*_sT_sL^s\varphi,f_{UZ}(u,.)\right\rangle, \\
			V_n(u) & = \sqrt{n}\left\langle L^{-s}(\alpha_n I + T^*_sT_s)^{-1}T_s^*T_sL^s\varphi - \varphi,f_{UZ}(u,.)\right\rangle .
		\end{aligned}
	\end{equation*}
	We show below that $\|II_n+III_n+IV_n+V_n\|_\infty = o_P(1)$. To that end, first since $T_s=TL^{-s}$
	\begin{equation*}
		\begin{aligned}
			\|II_n\|_\infty & = \sqrt{n}\sup_u\left\langle(\alpha_n I + T^*T)^{-1}(\hat T^* - T^*)(\hat r - \hat T\varphi),f_{UZ}(u,.)\right\rangle \\
			& \leq \sqrt{n}\left\|(\hat T^* - T^*)(\hat r - \hat T\varphi)\right\|\sup_u\left\|(\alpha_n I + T^*T)^{-1}f_{UZ}(u,.)\right\| \\
			& \lesssim \sqrt{n}\|\hat T^* - T^*\|_{\mathrm{op}}\left\|\hat r - \hat T\varphi\right\|\left\|(\alpha_n I + T^*T)^{-1}T^*T\right\|_{\mathrm{op}} \\
			& \lesssim_P \sqrt{n}\left(\frac{1}{\sqrt{nh_n^{p+q}}} + h_n^t\right)\left(\frac{1}{\sqrt{nh_n^q}} + h_n^t\right) = o_P(1),
		\end{aligned}
	\end{equation*}
	where the third line follows under Assumptions~\ref{as:hilbert_scale} (i) and \ref{as:smoothness} (i); and the fourth by arguments as in the proof of Corollary~\ref{cor:rates_npiv} under Assumptions~\ref{as:dgp} and \ref{as:tuning} (ii).
	
	Second,
	\begin{equation*}
		\begin{aligned}
			\|V_n\|_\infty & = \sqrt{n}\sup_u\left|\left\langle L^{-(a+s)}\left[(\alpha_n I + T^*_sT_s)^{-1}T_s^*T_s - I\right]L^s\varphi,L^af_{UZ}(u,.) \right\rangle\right| \\
			& \lesssim \sqrt{n}\left\|T_s\left[(\alpha_n I + T^*_sT_s)^{-1}T_s^*T_s - I\right]L^s\varphi\right\| \\
			& \lesssim \sqrt{n}\left\|T_s\alpha_n(\alpha_n I + T^*_sT_s)^{-1}(T^*_sT_s)^\frac{b-s}{2(a+s)}\right\|_{\mathrm{op}} \\
			& \lesssim \sqrt{n}\sup_\lambda\left|\frac{\alpha_n\lambda^{\frac{b-s}{2(a+s)}+\frac{1}{2}}}{\alpha_n + \lambda}\right| \\
			& = \sqrt{n}\sup_\lambda\left|\frac{\alpha_n\lambda}{\alpha_n + \lambda}\right| \\
			& \lesssim \sqrt{n}\alpha_n = o(1),
		\end{aligned}
	\end{equation*}
	where the first equality follows since $L$ is self-adjoint; the second line by the Cauchy-Schwartz inequality since $\sup_u\|L^af_{UZ}(u,.)\|<\infty$ under Assumption~\ref{as:smoothness} (i) and by Assumption~\ref{as:hilbert_scale} (i); the third since $L^s\varphi=(T^*_sT_s)^\frac{b-s}{2(a+s)}\psi$ for some $\psi\in L_2$ by \cite{engl1996regularization}, Corollary 8.22; the fourth by the isometry of the functional calculus; and the last since $2s=b-a$ and since $n\alpha_n^2\to 0$ under Assumption~\ref{as:tuning} (iii).
	
	Next, decompose $III_n(u) = R_{1n}(u) + R_{2n}(u)$ with
	\begin{equation*}
		\begin{aligned}
			R_{1n}(u) & = \sqrt{n}\left\langle\hat r - \hat T\varphi, \hat T(\alpha_nI+\hat T^*\hat T)^{-1}\hat T^*(T - \hat T)(\alpha_nI + T^*T)^{-1}f_{UZ}(u,.) \right\rangle, \\
			R_{2n}(u) & = \sqrt{n}\left\langle\hat r - \hat T\varphi, \hat T(\alpha_nI+\hat T^*\hat T)^{-1}(T^* - \hat T^*)T(\alpha_nI + T^*T)^{-1}f_{UZ}(u,.) \right\rangle.
		\end{aligned}
	\end{equation*}
	By the Cauchy-Schwartz inequality and previous computations
	\begin{equation*}
		\begin{aligned}
			\|R_{1n}\|_\infty & \leq \sqrt{n}\left\|\hat r - \hat T\varphi\right\|\|\hat T(\alpha_nI+ \hat T^*\hat T)^{-1}\hat T^*\|_{\mathrm{op}}\|\hat T - T\|_{\mathrm{op}}\sup_u\|(\alpha_nI + T^*T)^{-1}f_{UZ}(u,.)\| \\
			& \lesssim  \sqrt{n}\left\|\hat r - \hat T\varphi\right\|\|\hat T - T\|_{\mathrm{op}}
		\end{aligned}
	\end{equation*}
	and
	\begin{equation*}
		\begin{aligned}
			\|R_{2n}\|_\infty & \leq \sqrt{n}\left\|\hat r - \hat T\varphi\right\|\|\hat T(\alpha_nI+ \hat T^*\hat T)^{-1}\|_{\mathrm{op}}\|\hat T^* - T^*\|_{\mathrm{op}}\sup_u\|T(\alpha_nI + T^*T)^{-1}f_{UZ}(u,.)\| \\
			& \lesssim \sqrt{n}\left\|\hat r - \hat T\varphi\right\|\alpha_n^{-1/2}\|\hat T - T\|_{\mathrm{op}}\|T(\alpha_nI + T^*T)^{-1}(T^*T)^{\kappa/2a}\|_{\mathrm{op}} \\
			& \lesssim \sqrt{n}\left\|\hat r - \hat T\varphi\right\|\|\hat T - T\|_{\mathrm{op}}\alpha_n^{\kappa/2a-1}.
		\end{aligned}
	\end{equation*}
	Therefore, under Assumption~\ref{as:tuning} since $\kappa>2a$
	\begin{equation*}
		\begin{aligned}
			\|III_n\|_\infty & \lesssim \sqrt{n}\left\|\hat r - \hat T\varphi\right\|\alpha_n^{-1/2}\|\hat T - T\|_{\mathrm{op}} \\
			& \sqrt{n}\left(\frac{1}{\sqrt{nh_n^{p+q}}} + h_n^t\right)\left(\frac{1}{\sqrt{nh_n^q}} + h_n^t\right) = o_P(1).
		\end{aligned}
	\end{equation*}
	
	Similarly, decompose $IV_n(u) = S_{1n}(u) + S_{2n}(u)$ with
	\begin{equation*}
		\begin{aligned}
			S_{1n}(u) & = \sqrt{n}\left\langle L^{-s}(\alpha_nI + \hat T^*_s\hat T_s)^{-1}\hat T^*_s(\hat T_s - T_s)\alpha_n(\alpha_nI + T^*_sT_s)^{-1}L^s\varphi,f_{UZ}(u,.) \right\rangle, \\
			S_{2n}(u) & = \sqrt{n}\left\langle L^{-s}(\alpha_nI + \hat T^*_s\hat T_s)^{-1}(\hat T_s^* - T_s^*)T_s\alpha_n(\alpha_nI + T^*_sT_s)^{-1}L^s\varphi,f_{UZ}(u,.) \right\rangle. \\
		\end{aligned}
	\end{equation*}
	Likewise, by the Cauchy-Schwartz inequality and previous computations
	\begin{equation*}
		\begin{aligned}
			\|S_{1n}\|_\infty & \lesssim \sqrt{n}\left\|T_s(\alpha_nI + \hat T^*_s\hat T_s)^{-1}\hat T^*_s\left(\hat T_s - T_s\right)\alpha_n(\alpha_nI + T^*_sT_s)^{-1}L^s\varphi\right\| \\
			& \lesssim_P \sqrt{n}\alpha_n^{1/2}\left\|\alpha_n(\alpha_nI + T^*_sT_s)^{-1}(T^*_sT_s)^{\frac{b-s}{2(a+s)}}\psi\right\|  \lesssim_P \sqrt{n}\alpha_n = o_P(1)
		\end{aligned}
	\end{equation*}
	and
	\begin{equation*}
		\begin{aligned}
			\|S_{2n}\|_\infty & \lesssim \sqrt{n}\left\|T_s(\alpha_nI + \hat T_s^*\hat T_s)^{-1}\left(\hat T_s^* - T_s^*\right)\alpha_nT_s(\alpha_nI + T^*_sT_s)^{-1}L^s\varphi\right\| \\
			& \lesssim_P \sqrt{n}\left\|\alpha_nT_s(\alpha_nI + T^*_sT_s)^{-1}L^s\varphi\right\| \lesssim_P \sqrt{n}\alpha_n  = o_P(1),
		\end{aligned}
	\end{equation*}
	where we use $\|\hat T - T\|_{\rm op}\lesssim_P\alpha_n^{1/2}$ and $2s=b-a$; see also the proof of Theorem~\ref{thm:risk_hs}. Therefore, $\|IV_n\|_\infty = o_P(1)$.
	
	Combining all estimates, we obtain uniformly over $u\in\R$
	\begin{equation}\label{eq:uniform_expansion}
		I_n(u) = \sqrt{n}\left\langle T^*(\hat r - \hat T\varphi),(\alpha_nI+T^*T)^{-1}f_{UZ}(u,.) \right\rangle + o_P(1).
	\end{equation}
	Next, note that
	\begin{equation*}
		(\hat r - \hat T\varphi)(w) = \frac{1}{n}\sum_{i=1}^n(Y_i - [\varphi\ast K_z](Z_i))h_n^{-q}K_w\left(h_n^{-1}(W_i - w)\right)
	\end{equation*}
	with $[\varphi\ast K_z](z) \triangleq \int \varphi(v)h_n^{-p}K_z\left(h_n^{-1}(z-v)\right)\dx v$, whence
	\begin{equation*}
		T^*(\hat r - \hat T\varphi) = \frac{1}{n}\sum_{i=1}^n(Y_i - [\varphi\ast K_z](Z_i))[f_{ZW}\ast K_w](.,W_i)
	\end{equation*}
	with $[f_{ZW}\ast K_w](z,w) \triangleq \int f_{ZW}(z,v)h_n^{-q}K_w\left(h_n^{-1}(w-v)\right)\dx v$. Using this observation, decompose equation~(\ref{eq:uniform_expansion}) further
	\begin{equation*}
		\begin{aligned}
			I_n(u) & = \frac{1}{\sqrt{n}}\sum_{i=1}^nU_i\left\langle f_{ZW}(.,W_i),(T^*T)^{-1}f_{UZ}(u,.)\right\rangle + Q_{1n} + Q_{2n} + Q_{3n} + o_P(1) \\
		\end{aligned}
	\end{equation*}
	with
	\begin{equation*}
		\begin{aligned}
			Q_{1n}(u) & = \left\langle\frac{1}{\sqrt{n}}\sum_{i=1}^n[\varphi - \varphi\ast K_z](Z_i)[f_{ZW}\ast K_w](.,W_i), (\alpha_n I + T^*T)^{-1}f_{UZ}(u,.)\right\rangle, \\
			Q_{2n}(u) & = \left\langle\frac{1}{\sqrt{n}}\sum_{i=1}^nU_i\left\{[f_{ZW}\ast K_w](.,W_i) - f_{ZW}(.,W_i)\right\}, (\alpha_n I + T^*T)^{-1}f_{UZ}(u,.)\right\rangle, \\
			Q_{3n}(u) & = \left\langle\frac{1}{\sqrt{n}}\sum_{i=1}^nU_if_{ZW}(.,W_i),\left[(\alpha_n I + T^*T)^{-1} - (T^*T)^{-1}\right]f_{UZ}(u,.)\right\rangle.
		\end{aligned}
	\end{equation*}	
	By the Cauchy-Schwartz inequality
	\begin{equation*}
		\begin{aligned}
			\E\|Q_{1n}\|_\infty & \leq \E\left\|\frac{1}{\sqrt{n}}\sum_{i=1}^n[\varphi - \varphi\ast K_z](Z_i)[f_{ZW}\ast K_w](.,W_i)\right\|\sup_u\|(\alpha_nI + T^*T)^{-1}f_{UZ}(u,.) \| \\
			& \lesssim  \frac{1}{\sqrt{n}}\sum_{i=1}^n\E|[\varphi - \varphi\ast K_z](Z_i)|\left\|[f_{ZW}\ast K_w](.,W_i)\right\|\\
			& \leq \sqrt{n}\|\varphi - \varphi\ast K_z\|\|f_{ZW}\ast K_w\| \\
			& \lesssim \sqrt{n}h_n^{b}, \\
		\end{aligned}
	\end{equation*}
	where the second line follows by triangle inequality and Assumption~\ref{as:smoothness} (i); the third by Assumption~\ref{as:dgp} (i), Cauchy-Schwartz inequality, and since $f_Z$ and $f_W$ are uniformly bounded under Assumption~\ref{as:dgp} (ii); and the last by the standard bias computations under Assumptions~\ref{as:hilbert_scale} (ii) and \ref{as:dgp}, and Young's inequality under Assumption~\ref{as:dgp} (ii) and (iv).
	
	Similarly, by the Cauchy-Schwartz inequality and Assumption~\ref{as:hilbert_scale} (i)
	\begin{equation*}
		\begin{aligned}
			\E\|Q_{2n}\|_\infty^2 & \lesssim \E\left\|\frac{1}{\sqrt{n}}\sum_{i=1}^nU_i\left\{[f_{ZW}\ast K_w](.,W_i) - f_{ZW}(.,W_i)\right\}\right\|^2 \\
			& = \E\left\|U\left\{[f_{ZW}\ast K_w](.,W) - f_{ZW}(.,W)\right\}\right\|^2 \\
			& \lesssim \E\|[f_{ZW} - f_{ZW}\ast K_w](.,W)\|^2 \\
			& \lesssim \|f_{ZW} - f_{ZW}\ast K_w\|^2 \\
			& \lesssim h_n^{2t},
		\end{aligned}
	\end{equation*}
	where the second line follows under the i.i.d. assumption; the third since $\E[U|W]\leq C$ under Assumption~\ref{as:dgp} (i); the fourth since $f_W$ is uniformly bounded under Assumption~\ref{as:dgp} (ii); and the last by the standard bias computations under Assumptions~\ref{as:hilbert_scale} (ii) and \ref{as:dgp} (iv).
	
	Lastly, by the Cauchy-Schwartz inequality
	\begin{equation*}
		\begin{aligned}
			\E\|Q_{3n}\|_\infty^2 & = \E\left\|\frac{1}{\sqrt{n}}\sum_{i=1}^nU_if_{ZW}(.,W_i)\right\|^2\sup_u\left\|\left[(\alpha_n I + T^*T)^{-1} - (T^*T)^{-1}\right]f_{UZ}(u,.)\right\|^2 \\
			& = \E\|U_if_{ZW}(.,W_i)\|^2 \sup_u\left\|\alpha_n(\alpha_n I + T^*T)^{-1}(T^*T)^{-1}f_{UZ}(u,.)\right\|^2 \\
			& \lesssim \left\|\alpha_n(\alpha_n I + T^*T)^{-1}(T^*T)^{\kappa/2a-1}\right\|_{\mathrm{op}} \\
			& \lesssim \sup_\lambda\left|\frac{\alpha_n\lambda^{\kappa/2a-1}}{\alpha_n + \lambda}\right| \lesssim \alpha_n^{(\kappa/2a-1)\wedge 1},
		\end{aligned}
	\end{equation*}
	where the second inequality follows under Assumptions~\ref{as:dgp} (i); the third line under Assumptions~\ref{as:hilbert_scale}, \ref{as:dgp} (i)-(ii), and \ref{as:smoothness} (i); and the last by the isometry of functional calculus.
	
	Combining these estimates under Assumptions~\ref{as:smoothness} (i) and \ref{as:tuning}, we obtain the result
	\begin{equation*}
		\begin{aligned}
			II_n(u) & = \frac{1}{\sqrt{n}}\sum_{i=1}^nU_i\left\langle f_{ZW}(.,W_i),(T^*T)^{-1}f_{UZ}(u,.)\right\rangle + O_P\left(\sqrt{n}h_n^b + h_n^t + \alpha_n^{(\kappa/2a-1)\wedge 1}\right) + o_P(1) \\
			& = \frac{1}{\sqrt{n}}\sum_{i=1}^nU_i[T(T^*T)^{-1}f_{UZ}(u,.)](W_i) + o_P(1).
		\end{aligned}
	\end{equation*}
\end{proof}

\section{Additional Monte Carlo experiments}
In this section, we report results of additional Monte Carlo experiments when the structural function is $\varphi(x) = \exp(-x^2/4)$. The rest of the data-generating process is the same as in the main part of the paper.

Figure~\ref{fig:3} shows the distribution of the test statistics under the null hypothesis and the two alternative hypotheses for different sample sizes. The two distributions are sufficiently distinct once the alternative hypothesis becomes more separated from the null hypothesis.
\begin{figure}
	\centering
	\begin{subfigure}[b]{0.49\textwidth}
		\includegraphics[width=\textwidth]{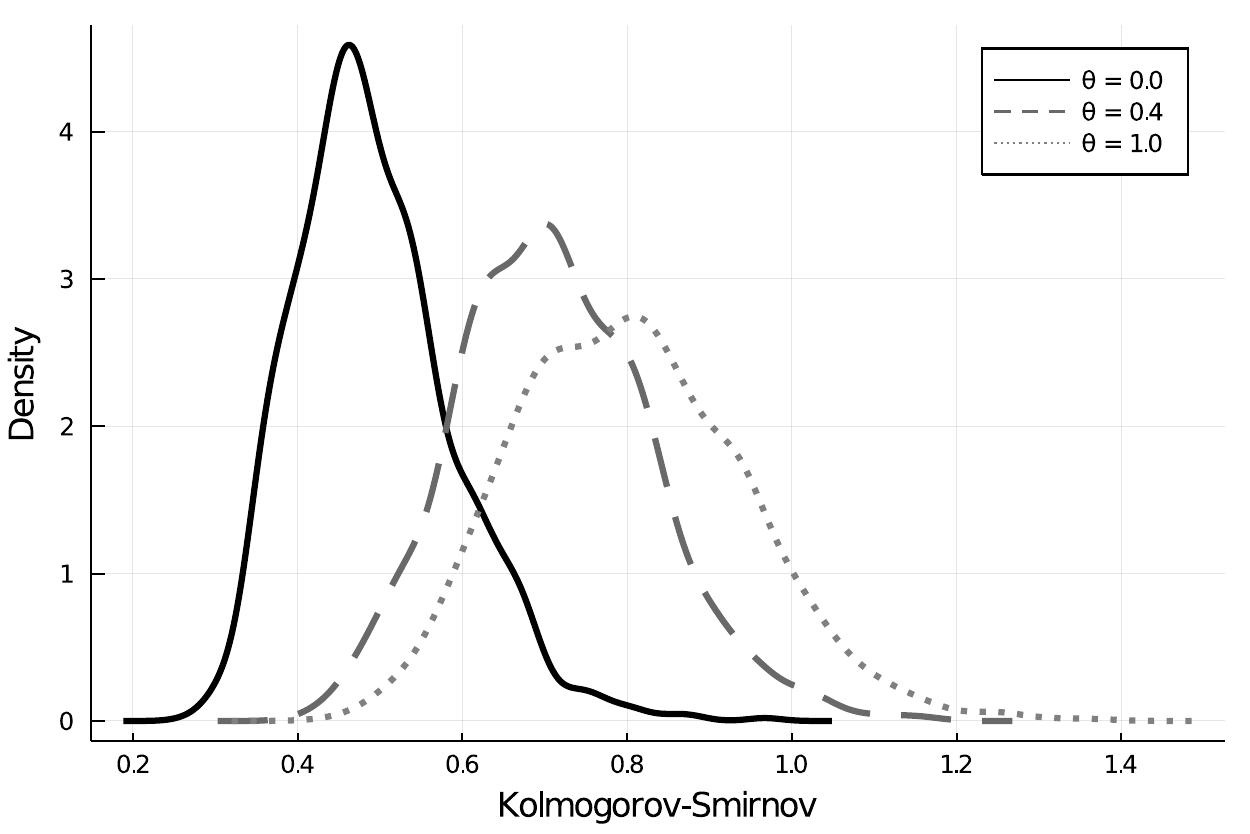}
		\caption{Sample size: $n=500$}
	\end{subfigure}
	\begin{subfigure}[b]{0.49\textwidth}
		\includegraphics[width=\textwidth]{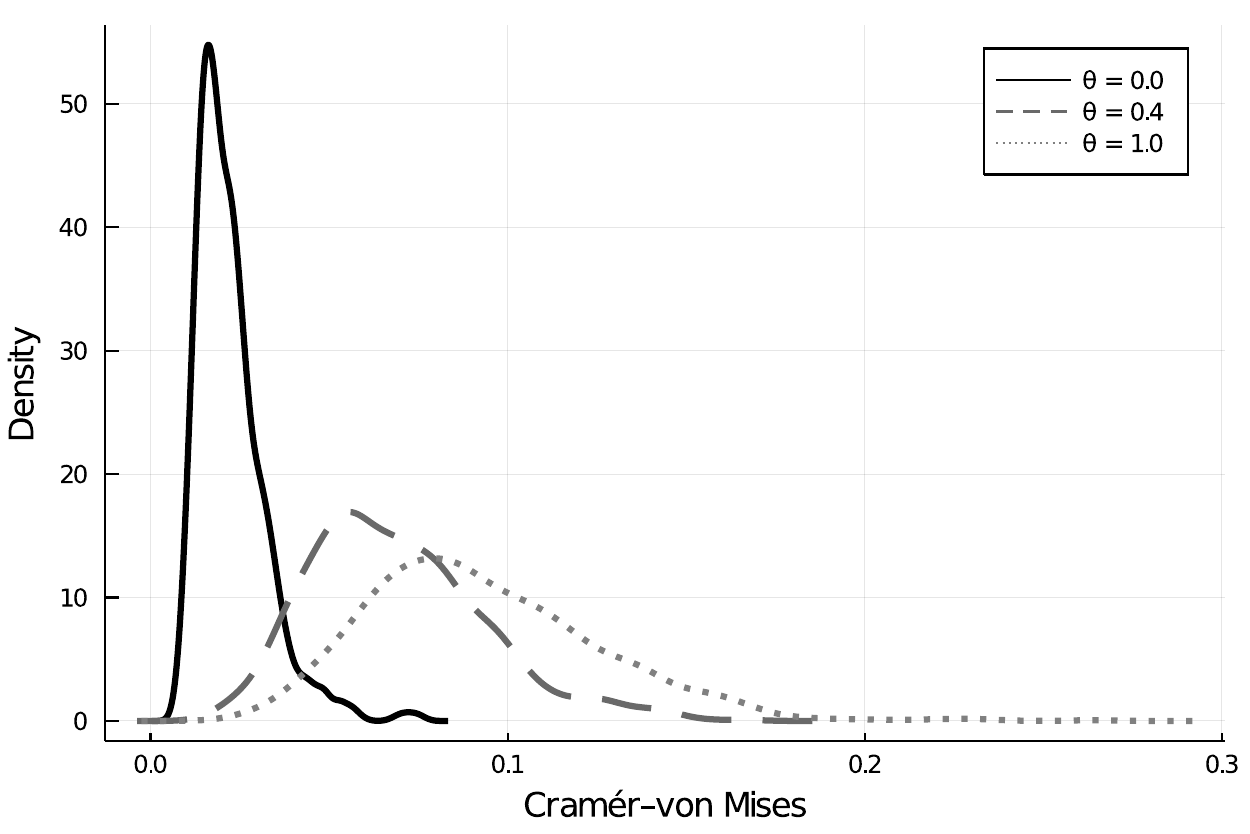}
		\caption{Sample size: $n=500$}
	\end{subfigure}
	\begin{subfigure}[b]{0.49\textwidth}
		\includegraphics[width=\textwidth]{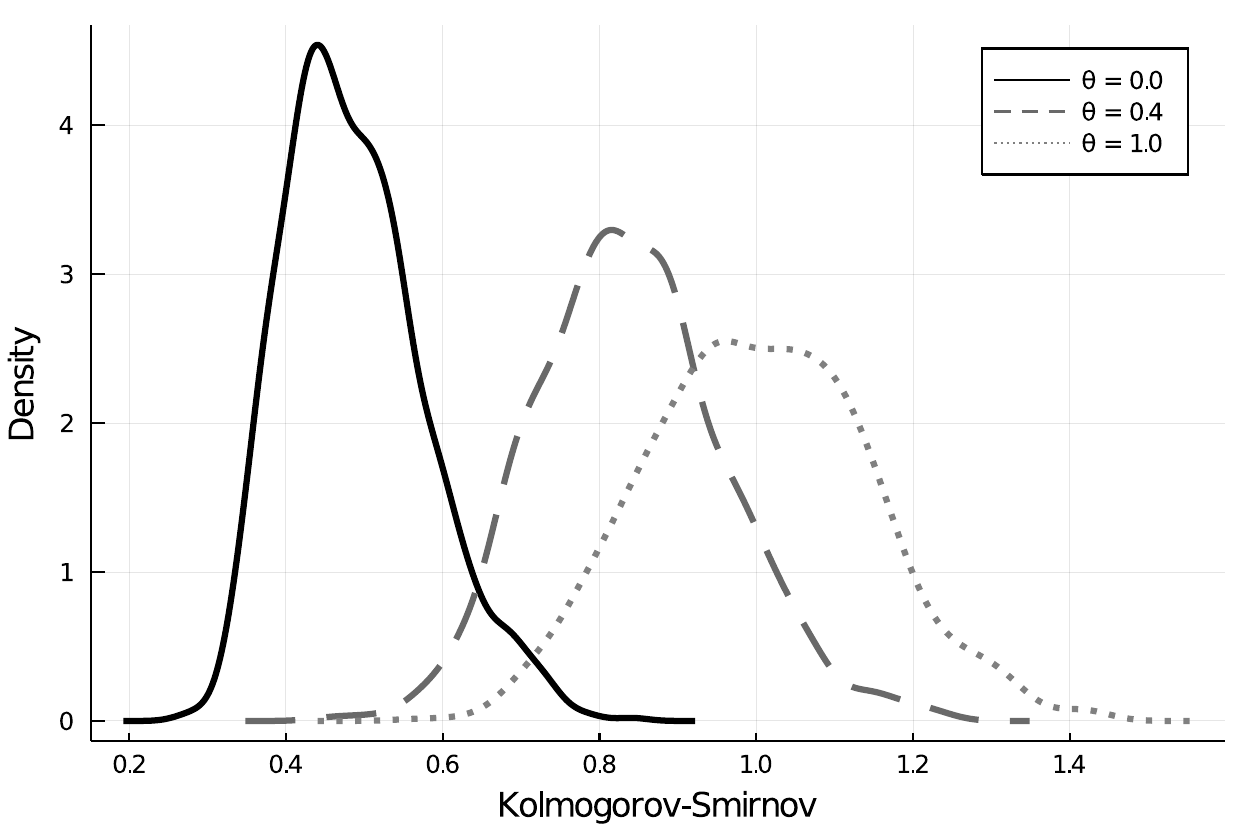}
		\caption{Sample size: $n=1,000$}
	\end{subfigure}
	\begin{subfigure}[b]{0.49\textwidth}
		\includegraphics[width=\textwidth]{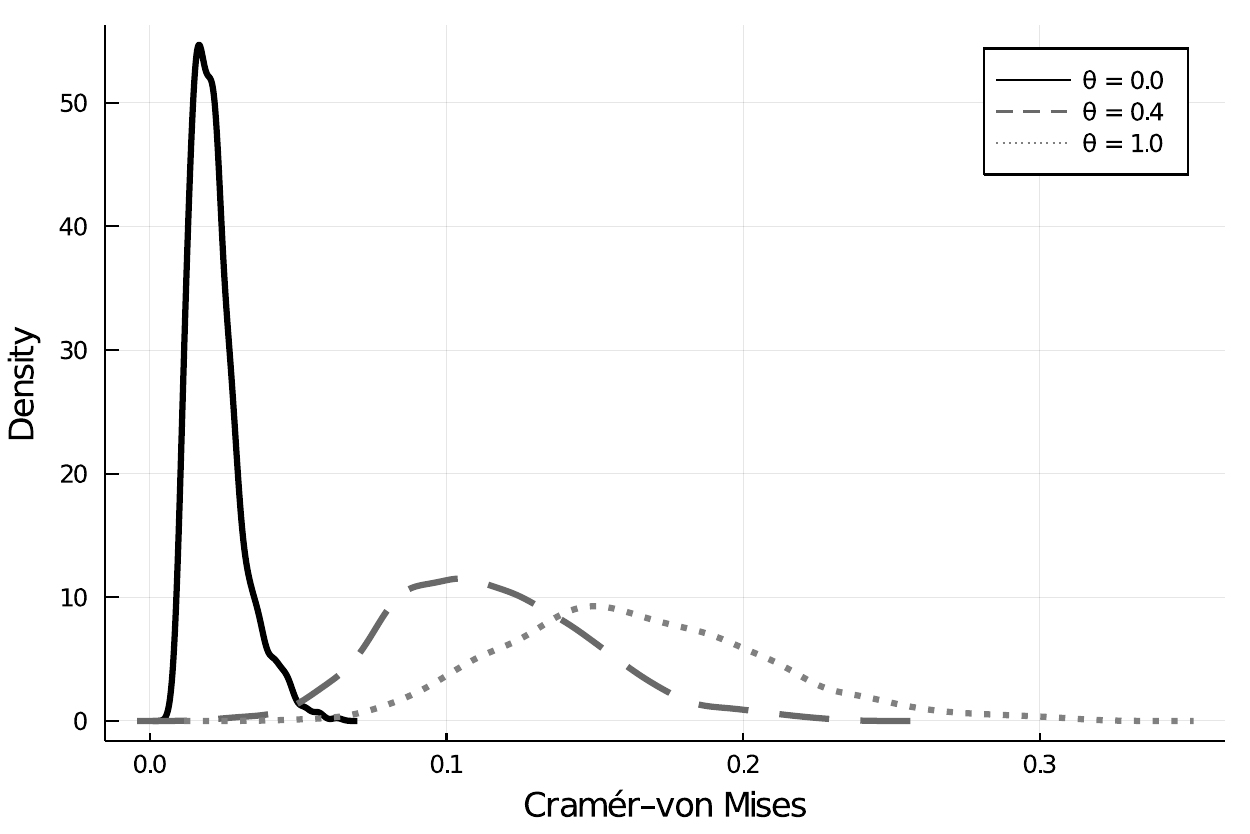}
		\caption{Sample size: $n=1,000$}
	\end{subfigure}
	\caption{Finite-sample distribution of the test -- density estimates of the distribution of Kolmogorov-Smirnov and Cram\'{e}r-von Mises statistics under $H_0$, $\theta=0$ (solid line), and two alternative hypotheses, $\theta=0.4$ (dashed line) and $\theta=1$ (dotted line).}
	\label{fig:3}
\end{figure}
We plot in Figure~\ref{fig:4} the power curves when the level of the test is fixed at $5\%$. The power of the test increases once alternative hypotheses become more distant from the null hypothesis and when the sample size is larger. The Cram\'{e}r-von Mises test seems to have a higher power for the  class of considered alternatives.
\begin{figure}
	\centering
	\begin{subfigure}[b]{0.49\textwidth}
		\includegraphics[width=\textwidth]{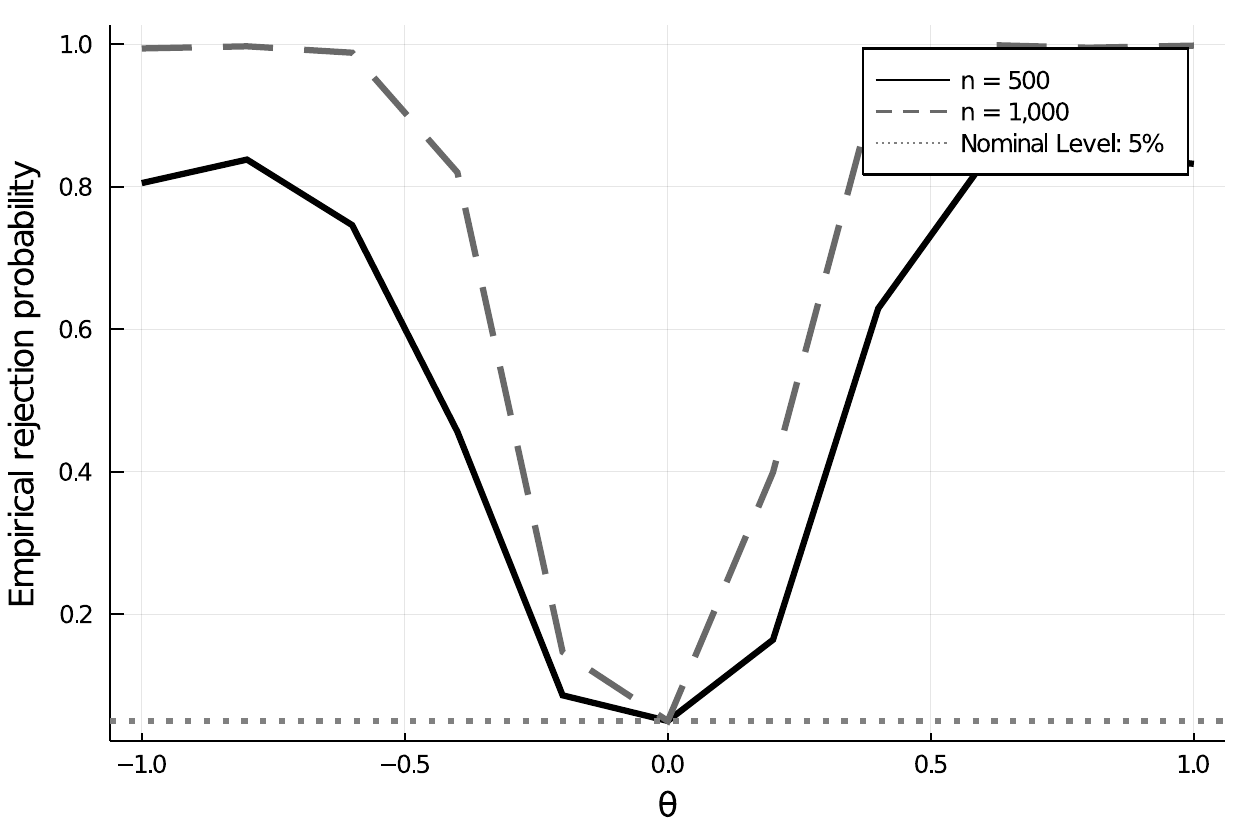}
		\caption{Kolmogorov-Smirnov test}
	\end{subfigure}
	\begin{subfigure}[b]{0.49\textwidth}
		\includegraphics[width=\textwidth]{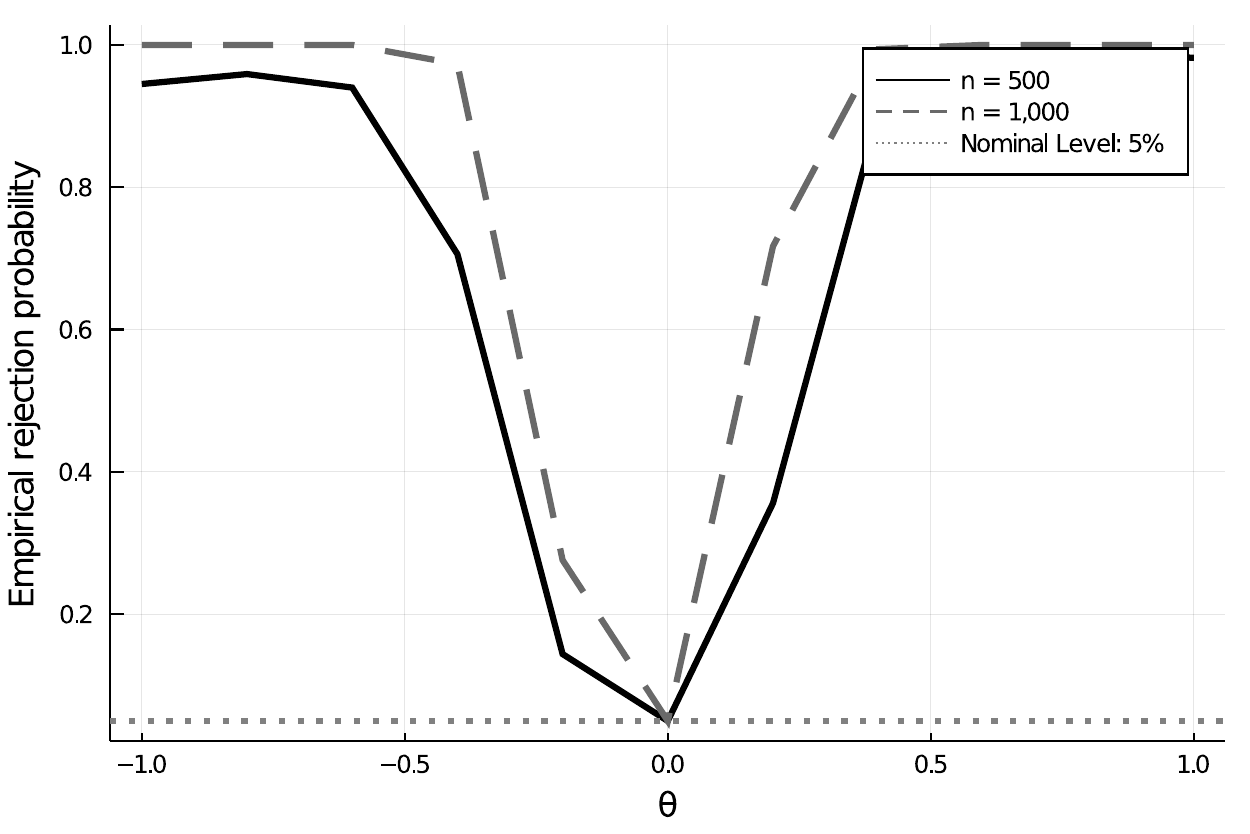}
		\caption{Cram\'{e}r-von Mises test}
	\end{subfigure}
	\caption{Power curves. The figure shows empirical rejection probabilities as a function of degree of separability $\theta$ for samples of size $n=500$ (solid line) and $n=1,000$ (dashed line). The value $\theta=0$ corresponds to the separable model, while $\theta\ne 0$ are deviations from separability. The nominal level of the test is set at $5\%$.}
	\label{fig:4}
\end{figure}
Overall, the findings are largely similar to the findings of experiments presented in the main part of the paper.

\clearpage
\bibliographystyle{econometrica}
\bibliography{references}

\end{document}